\documentclass[12pt]{amsart}

\usepackage{amsfonts, amsthm, amsmath}

\usepackage{rotating}

\usepackage{tikz}

\usepackage{graphics}

\usepackage{amssymb}

\usepackage{amscd}

\usepackage[latin2]{inputenc}

\usepackage{t1enc}

\usepackage[mathscr]{eucal}

\usepackage{indentfirst}

\usepackage{graphicx}

\usepackage{graphics}

\usepackage{pict2e}

\usepackage{mathrsfs}

\usepackage{enumerate}
\usepackage[pagebackref]{hyperref}
\usepackage{cite}
\usepackage{color}
\usepackage{epic}
\usepackage{hyperref} 
\usepackage{framed}
\usepackage{mathabx}
\usepackage{booktabs}
\usepackage{makecell}
\newcolumntype{V}{!{\vrule width 2pt}}

\numberwithin{equation}{section}
\def\blue{\textcolor{blue}}

\theoremstyle{plain}

\newtheorem{theorem}{Theorem}[section]

\newtheorem{lemma}[theorem]{Lemma}

\newtheorem{corollary}[theorem]{Corollary}

\newtheorem{proposition}[theorem]{Proposition}

\theoremstyle{definition}

\newtheorem{Def}[theorem]{Definition}

\newtheorem{Observation}[theorem]{Observation}

\newtheorem{conj}[theorem]{Conjecture}

\newtheorem{remark}[theorem]{Remark}

\newtheorem{?}[theorem]{Problem}

\newcommand{\N}{\mathbb{N}}
\newcommand{\Z}{\mathbb{Z}}
\newcommand{\RR}{\mathbb{R}}

\newcommand{\I}{\mathfrak{I}}

\def\des{\mathsf{des}}
\def\S{\mathfrak{S}}
\def\iar{\mathsf{iar}}
\def\lir{\mathsf{lir}}
\def\comp{\mathsf{comp}}
\def\DES{\mathrm{DES}}
\def\ASC{\mathrm{ASC}}
\def\LMAX{\mathrm{LMAX}}
\def\LMAXP{\mathrm{LMAXP}}
\def\LMIN{\mathrm{LMIN}}
\def\st{\mathsf{st}}
\def\asc{\mathsf{asc}}
\def\dd{\mathsf{dd}}

\def\del{\operatorname{del}}
\def\ins{\operatorname{ins}}
\def\AW{\mathcal{AW}}
\def\ics{\mathsf{ics}}
\def\equ{\mathsf{equ}}
\def\SP{\mathrm{SP}}
\def\DESB{\mathrm{DESB}}
\def\lmax{\mathsf{lmax}}
\def\lmin{\mathsf{lmin}}

\def\izero{\mathsf{izero}}
\def\da{\mathsf{da}}
\def\e{\mathfrak{a}}

\newcommand{\In}{\mathcal{I}}
\def\id{\mathrm{id}}
\def\DESB{\mathrm{DESB}}

\def\lmax{\mathsf{lmax}}
\def\ST{\mathrm{ST}}
\def\SE{\mathcal{S}}
\def\AVA{\mathrm{AVA}}
\def\ldes{\mathsf{ldes}}

\begin{document}

\title[Refined Wilf-equivalences]{Refined Wilf-equivalences by Comtet statistics}

\author[S. Fu]{Shishuo Fu}
\address[Shishuo Fu]{College of Mathematics and Statistics, Chongqing University, Huxi campus, Chongqing 401331, P.R. China}
\email{fsshuo@cqu.edu.cn}

\author[Z. Lin]{Zhicong Lin}
\address[Zhicong Lin]{Research Center for Mathematics and Interdisciplinary Sciences, Shandong University, Qingdao 266237, P.R. China}
\email{linz@sdu.edu.cn}

\author[Y. Wang]{Yaling Wang}
\address[Yaling Wang]{College of Mathematics and Statistics, Chongqing University, Huxi campus, Chongqing 401331, P.R. China}
\email{wyl032021@163.com}

\date{\today}

\begin{abstract}
We launch a systematic study of the refined Wilf-equivalences by the statistics $\comp$ and $\iar$, where  $\comp(\pi)$ and $\iar(\pi)$ are  the number of components and the length of the initial ascending run of a permutation $\pi$, respectively. As Comtet was the first one to consider the statistic $\comp$ in his book {\em Analyse combinatoire},  any statistic equidistributed with $\comp$ over a class of permutations is called by us a {\em Comtet statistic} over such class. This work is motivated by a triple equidistribution result of Rubey on $321$-avoiding permutations, and a recent result of the first and third authors that $\iar$ is a Comtet statistic over separable permutations. Some highlights of our results are: 
\begin{itemize}
	\item Bijective proofs of the symmetry of the double Comtet distribution $(\comp,\iar)$ over several Catalan and Schr\"oder classes, preserving the values of the left-to-right maxima.
	\item A complete classification of $\comp$- and $\iar$-Wilf-equivalences for length $3$ patterns and pairs of length $3$ patterns. Calculations of the $(\des,\iar,\comp)$ generating functions over these pattern avoiding classes and separable permutations.
	\item A further refinement by the Comtet statistic $\iar$, of Wang's recent descent-double descent-Wilf equivalence between separable permutations and $(2413,4213)$-avoiding permutations.
\end{itemize}
\end{abstract}



\maketitle


\section{Introduction}\label{sec1: intro}

A permutation $\pi=\pi(1)\cdots\pi(n)\in\S_n$, the symmetric group on $[n]:=\{1,2,\ldots,n\}$, is said to {\em avoid} the permutation (or {\em pattern}) $\sigma=\sigma(1)\cdots \sigma(k)\in\S_k$, $k\le n$, if and only if there is no subsequence $\pi(j_1)\pi(j_2)\cdots\pi(j_k)$ with $j_1<j_2<\cdots<j_k$, such that $\pi(j_a)<\pi(j_b)$ if and only if $\sigma(a)<\sigma(b)$ for all $1\le a<b\le k$. Otherwise, we say that the permutation $\pi$ {\em contains the pattern $\sigma$}. 

The notion of permutation pattern was introduced by Knuth \cite[pp.~242-243]{knu} in 1968, but studied intensively and systematically for the first time by Simion and Schmidt \cite{SS} in 1985. Ever since then, it has become an active and prosperous research subject. The reader is referred to two book expositions, \cite[Chapters 4 and 5]{bon} and \cite{kit}, on this topic, as well as the numerous references therein. In the early 1980s, Herbert Wilf posed the problem of identifying equirestrictive sets of forbidden patterns. Let $P$ be a (finite) collection of patterns and $\mathcal{W}$ a set of permutations, we write $\mathcal{W}(P)$ for the set of all permutations in $\mathcal{W}$ that avoid simultaneously every pattern contained in $P$. We will say, as it has become a standard terminology, that two sets of patterns, $P$ and $Q$, are {\em Wilf-equivalent}, denoted by $P\sim Q$, if $|\S_n(P)|=|\S_n(Q)|$ for all positive integers $n$. 


In this paper, we will restrict ourselves to the case where $|P|=|Q|\le 2$, and the lengths of the patterns in $P$ and $Q$ are no greater than $4$. Once two sets of patterns $P$ and $Q$ are known to be Wilf-equivalent, a natural direction to go deeper, is to make further restrictions on these $P$- or $Q$-avoiding permutations, and to see if the equinumerosity still holds. One such restriction is to consider the enumeration refined by various permutation statistics. In general, a {\em statistic} on a set of objects $S$ is simply a function from $S$ to $\N:=\{0,1,2,\ldots\}$. A {\em set-valued statistic} on $S$ is a function from $S$ to the set of finite subsets of $\N$. Given a permutation $\pi\in\S_n$, we mainly consider the set-valued statistic
$$\DES(\pi):=\{i\in[n-1]:\pi(i)<\pi(i+1)\},$$
called the {\em descent set} of $\pi$, and two statistics
$$\des(\pi):=|\DES(\pi)|\quad \text{and}\quad \iar(\pi):=\min(\DES(\pi)\cup\{n\}),$$
called the {\em {\bf\em des}cent number} and the {\em {\bf\em i}nitial {\bf\em a}scending {\bf\em r}un} of $\pi$, respectively. Clearly, $\iar(\pi)$ can also be interpreted as the position of the leftmost descent of $\pi$, which indicates that $\iar$ is determined by $\DES$. It should be noted that $\iar$ was also called $\lir$, meaning ``leftmost increasing run'', in the literature (see e.g.~\cite{CK}). The statistic $\des$ is  known as an {\em Eulerian statistic} since its distribution over $\S_n$ is the {\em$n$-th  Eulerian polynomial} 
$$
A_n(t):=\sum_{\pi\in\S_n}t^{\des(\pi)}.
$$ 

Another statistic highlighted in our study is $\comp(\pi)$, which can be introduced as
$$
\comp(\pi):=|\{i: \forall j\leq i,\: \pi(j)\leq i\}|. 
$$
It is equal to the maximum number of {\em components} (see \cite{adin,CK,CKS}) in an expression of $\pi$ as a {\em direct sum} of permutations. For instance, $\comp(312465)=3$, the three components being $312$, $4$, and $65$ and $312465=312\oplus1\oplus 21$ (see Sec.~\ref{sec:dirsum} for the definition of direct sum $\oplus$).
The statistic  $\comp$ dates back at least to Comtet~\cite[3, Ex.~VI.14]{com}, who proved the generating function for the number $f(n)$ of permutations of length $n$ with one component, also known as {\em indecomposable permutations}, to be 
$$
\sum_{n\geq1}f(n)z^n=1-\frac{1}{\sum_{n\geq0}n!z^n}. 
$$
Thus, any statistic equidistributed with $\comp$ over a class of restricted permutations will be called by us a {\em Comtet statistic} over such class.  The enumeration of pattern avoiding indecomposable permutations was  carried out by Gao, Kitaev and Zhang~\cite{GKZ}. It should be noted that $\iar$ and $\comp$ are not equidistributed over $\S_4$. Nonetheless, two of the authors~\cite{FW} proved  that $\iar$ is a Comtet statistic over {\em separable permutations}, the class of $(2413,3142)$-avoiding permutations. It is this result that motivates us to investigate systematically  the refined Wilf-equivalences by these two Comtet statistics, and sometimes jointly with other statistics. 

For a (possibly set-valued) statistic $\st$ on $\S_n$, we say two sets of patterns $P$ and $Q$ are {\em $\st$-Wilf-equivalent}, denoted as $P\sim_{\st}Q$, if for all positive integers $n$, we have
$$|\S_n(P)^{\st}|=|\S_n(Q)^{\st}|,$$
meaning that for a fixed value of $\st$, there are as many preimages in $\S_n(P)$ as those in $\S_n(Q)$. Note that by their definitions, $P\sim_{\DES}Q$ immediately implies $P\sim_{\iar}Q$ and $P\sim_{\des}Q$, but not conversely. The above refined Wilf-equivalence by one statistic can be naturally extended to the joint distribution of several permutation statistics, regardless of numerical or set-valued types. So expression like $P\sim_{(\DES,\comp)}Q$ and $|\S_n(P)^{\DES,\comp}|=|\S_n(Q)^{\DES,\comp}|$ should be understood well. It should be noted that refined Wilf-equivalences have already been extensively studied during the last two decades (see e.g.~\cite{CK,DDJSS,EP,kit,LK}). Especially, the focus of Dokos, Dwyer, Johnson, Sagan and Selsor~\cite{DDJSS} was on the refined Wilf-equivalences by Eulerian and Mahonian statistics. Hopefully with the results we present in this paper, one is convinced that considering the refinements by Comtet statistics is equally meaningful. 

Some highlights of our results will be outlined below. Before stating them, we need to recall some classical permutation statistics. For a permutation  $\pi\in\S_n$, we introduce 
\begin{align*}
\LMAX(\pi) &:=\{\pi(i)\in[n]:\pi(j)<\pi(i),\: \forall 1\le j < i\} \,\,\; \text{and}\\
\LMAXP(\pi) &:=\{i\in[n]:\pi(j)<\pi(i),\: \forall 1\le j < i\},
\end{align*}
the set of {\em values} and {\em positions of the left-to-right maxima} of $\pi$, respectively. The sets of values/positions of the {\em left-to-right minima}, the {\em right-to-left maxima} and the {\em right-to-left minima} of $\pi$ can be defined and denoted similarly if needed. 
We use lowercase letters to denote the cardinality of these sets, so for example,  $\LMIN(\pi)$ is the set of values of the left-to-right minima of $\pi$ and $\lmin(\pi)$ is the corresponding numerical statistic. We will also consider the set of {\em descent bottoms} of $\pi$
$$\DESB(\pi):=\{\pi(i+1)\in[n-1]: i\in\DES(\pi)\},$$
 which is another set-valued extension of $\des$ different from $\DES$. 
 
 The first one of our main results concerns a single pattern of length $3$.

\begin{theorem}\label{main:thm1}
For every $n\ge 1$, 
\begin{itemize}
\item[(i)] the two triples $(\LMAX,\iar,\comp)$ and $(\LMAX,\comp,\iar)$ have the same distribution over $\S_n(321)$; 
\item[(ii)] the two quadruples $(\LMAX,\DESB,\iar,\comp)$ and $(\LMAX,\DESB,\comp,\iar)$ have the same distribution over $\S_n(312)$;
\item[(iii)] the quadruples $(\LMAX, \LMIN, \iar, \comp)$ and $(\LMAX, \LMIN, \comp, \iar)$ have the same distribution over $\S_n(132)$.
\end{itemize}
\end{theorem}

The result on the symmetry of $(\comp,\iar)$ was inspired by several works in the literature. First of all, Theorem~\ref{main:thm1} (i) is essentially equivalent to a result of Rubey~\cite{rub} up to some elementary transformations on permutations. Details will be given in Sec.~\ref{sym:3.1}. Furthermore, Rubey's result is a symmetric generalization of an equidistribution due to Adin, Bagno and Roichman~\cite{adin}, which implies the {\em Schur-positivity} of the class of $321$-avoiding permutations with a prescribed number of components. 

Next, Claesson, Kitaev and Steingr\'imsson~\cite[Thm~2.2.48]{kit} constructed a bijection between separable permutations of length $n+1$ with $k+1$ components and Schr\"oder paths of order $n$ with $k$ horizontals at $x$-axis. Combining this bijection with the work in~\cite{FW} justifies $\iar$ being a Comtet statistic on separable permutations. It then follows from our Lemma~\ref{general:sym}, a general lemma proved in Sec.~\ref{sec:dirsum}, that we have the following symmetric double Comtet distribution. 
\begin{corollary}\label{sym:separ}
The double Comtet statistics $(\comp,\iar)$ is symmetric on separable permutations. 
\end{corollary}

Our second main result announced below  is a far reaching refinement of  Corollary~\ref{sym:separ}.

\begin{theorem}\label{thm:sep:sym}
There exists an involution  on $\S_n(2413,3142)$ that preserves the pair of set-valued statistics $(\LMAX,\DESB)$ but exchanges the pair $(\comp,\iar)$. Consequently,
$$
\sum_{\pi\in\S_n(2413,3142)}s^{\comp(\pi)}t^{\iar(\pi)}{\bf x}^{\LMAX(\pi)}{\bf y}^{\DESB(\pi)}=\sum_{\pi\in\S_n(2413,3142)}s^{\iar(\pi)}t^{\comp(\pi)}{\bf x}^{\LMAX(\pi)}{\bf y}^{\DESB(\pi)}
$$
where ${\bf x}^S:=\prod_{i\in S}x_i$ and ${\bf y}^S:=\prod_{i\in S}y_i$ for any subset $S\subseteq[n]$.  
\end{theorem}

The proof of Theorem~\ref{main:thm1} provided in Sec.~\ref{sec:one 3-pattern} is via two involutions on permutations that actually imply the even stronger symmetric phenomenon, namely the corresponding distribution matrices are {\em Hankel}; see Theorems~\ref{thm:gf-132} and \ref{thm:gf-132-312:132-321}. The proof of Theorem~\ref{thm:sep:sym} is based on a combinatorial bijection on the so-called {\em di-sk trees} introduced in~\cite{flz}. This bijection will also provide an alternative approach to Theorem~\ref{main:thm1}(ii). The details will be given in~\cite{flw2}. 

\begin{remark}  Rubey's bijective proof of a slight modification (see~Theorem~\ref{thm:Rubey}) of Theorem~\ref{main:thm1}(i) is via Dyck paths and the proof of Theorem~\ref{thm:sep:sym} that will appear in~\cite{flw2} is based on di-sk trees. Our bijective and unified proof of Theorem~\ref{main:thm1}(i)(ii), constructed directly on permutations, gets more insights into the symmetry of the double Comtet statistics, and therefore, it seems more likely to be extended to deal with such equidistributions over other bigger classes of pattern-avoiding permutations. 
\end{remark}

Our third main result shows how $\iar$, combined with $\des$ and the number of double descents would refine known results and imply new ones concerning separable and $(2413,4213)$-avoiding permutations. Interestingly, it does refine a nice $\gamma$-positivity interpretation for separable permutations~\cite{flz,lin} due to Zeng and the first two authors that we review below. 

Recall that a polynomial in $\RR[t]$ of degree $n$ is said to be {\em $\gamma$-positive} if it can be written as a linear combination of
$$
\{t^k(1 + t)^{n-2k}\}_{0\leq k\leq n/2}
$$
with non-negative coefficients. Many polynomials arising from  combinatorics and discrete geometry have been shown to be $\gamma$-positive; see the comprehensive  survey by Athanasiadis~\cite{ath}. One typical example is  the  Eulerian polynomials
$$
A_n(t)=\sum_{\pi\in\S_n}t^{\des(\pi)}=\sum_{k=0}^{\lfloor\frac{n-1}{2}\rfloor}|\Gamma_{n,k}|t^k(1 + t)^{n-1-2k},
$$
where $\Gamma_{n,k}$ is the set of permutations in $\S_n$ with $k$ descents and without double descents. Here an index
$i\in[n]$ is called a {\em double descent} of a permutation  $\pi\in\S_n$ if $\pi(i-1)>\pi(i)>\pi(i+1)$, where we use the convention $\pi(0)=\pi(n+1)=0$. The number of double descents of $\pi$ will be denoted as $\dd(\pi)$. This classical result is due to Foata and Sch\"utzenberger~\cite[Theorem~5.6]{FS} and has been extended in several different directions (cf.~\cite{ath}) in recent years. In particular, the first two authors together with Zeng~\cite{flz,lin} proved an analog for the descent polynomial over separable permutations 
\begin{equation}\label{gam:sch}
S_n(t):=\sum_{\pi\in\S_n(2413,3142)}t^{\des(\pi)}=\sum_{k=0}^{\lfloor\frac{n-1}{2}\rfloor}|\Gamma_{n,k}(2413,3142)|t^k(1 + t)^{n-1-2k}.
\end{equation}
In a recent work~\cite{LK} of the second author and Kim, they proved that $(2413,3142)\sim_{\des}(2413,4213)$ (see~\cite[Thm.~5.1]{LK}), and that  the $\gamma$-coefficient of the descent polynomial over $(2413,4213)$-avoiding permutations is analogously given by $|\Gamma_{n,k}(2413,4213)|$ (see~\cite[Eq.~(4.10)]{LK}). In view of \eqref{gam:sch}, we see the number of separable permutations of $[n]$ with $k$ descents and without double descents, is the same as that of $(2413,4213)$-avoiding permutations of $[n]$ with $k$ descents and without double descents. With this in mind, our third main result given below can be viewed as a refinement.

\begin{theorem} \label{thm:sep}
For $n\geq1$,
\begin{equation}\label{eq:sep}
\sum_{\pi\in\S_n(2413,3142)}t^{\des(\pi)}x^{\dd(\pi)}y^{\iar(\pi)}=\sum_{\pi\in\S_n(2413,4213)}t^{\des(\pi)}x^{\dd(\pi)}y^{\iar(\pi)}.
\end{equation} 
\end{theorem}

Theorem~\ref{thm:sep} refines Wang's equidistribution~\cite[Thm.~1.5]{wan} by the Comtet statistic $\iar$ and has many interesting consequences as can be found in Sections~\ref{sec:schroder} and~\ref{sec:revisit}. More detailed motivation that leads us to discover Theorem~\ref{thm:sep} will also be provided  in Sec.~\ref{sec:revisit}. 
Our proof of Theorem~\ref{thm:sep} in Sec.~~\ref{sec:revisit} is purely algebraic and finding a bijective proof remains open. 

Besides the above three main results, we will also calculate the joint distribution of $(\des,\iar,\comp)$ over permutations avoiding a set $P$ of patterns, where $P$ is taken to be a single pattern of length $3$, a pair of patterns of length $3$, as well as the three pairs $(2413,3142)$, $(2413,4213)$, and $(3412,4312)$, respectively. All the generating functions for these patterns turn out to be either algebraic or rational (see Tables~\ref{one 3-avoider} and~\ref{two 3-avoiders}), and as applications, complete classification of the $\iar$- or $\comp$-Wilf equivalences for these patterns is given. Moreover, our attempt to characterize the pattern pairs of length $4$ which are $(\iar,\comp)$-Wilf-equivalent to $(2413,3142)$ leads to Conjecture~\ref{schroder:iar}, which we have verified in some important cases. 

The rest of this paper is organized as follows. In Section~\ref{sec:notation}, we review some notations and  terminology and prove two general lemmas concerning the direct sum operation of permutations. The classification of refined Wilf-equivalences for a single pattern of length $3$ is carried out in Section~\ref{sec:one 3-pattern}, where the proof of Theorem~\ref{main:thm1} is provided as well. Section~\ref{sec3: two 3-patterns} is devoted to the investigation of pattern pairs of length $3$, while Section~\ref{sec:schroder} aims to characterize the pattern pairs of length $4$ that are $(\iar,\comp)$-Wilf-equivalent to $(2413,3142)$. The proof of Theorem~\ref{thm:sep} is given in Section~\ref{sec:revisit}, where a new recurrence for the $021$-avoiding inversion sequences is also proved.

\section{Notations and preliminaries}\label{sec:notation}
\subsection{Elementary operations}
For a given permutation $\pi\in\S_n$, there are three fundamental symmetry operations on $\pi$:
\begin{itemize}
\item its {\em reversal} $\pi^{\mathrm{r}}\in\S_n$ is given by $\pi^{\mathrm{r}}(i)=\pi(n+1-i)$;
\item its {\em complement} $\pi^{\mathrm{c}}\in\S_n$ is given by $\pi^{\mathrm{c}}(i)=n+1-\pi(i)$; 
\item its {\em inverse} $\pi^{-1}\in\S_n$, is the usual group theoretic inverse permutation. 
\end{itemize}
One thing we would like to point out, before we barge into classifying $\iar$-Wilf-equivalences for various patterns, is that by taking $\iar$ into consideration, we can no longer utilize the above three standard symmetries for permutations, since none of them preserves the length of the initial ascending run of $\pi$, when $n\ge 2$.  For the classical Wilf-equivalence, these symmetries reduce the number of possible equivalence classes considerably, since for example, $\pi$ avoids $213$ if and only if $\pi^{\mathrm{r}}$ avoids $312$. This fact about the statistic $\iar$ explains, at least partially, the following observations. 
\begin{Observation}\hfill
\begin{enumerate}
	\item The $\iar$-Wilf-equivalence is much less likely to be found than the Wilf-equivalence. 
	\item When $\iar$-Wilf-equivalence does hold, we cannot prove it using the three standard symmetries or their combinations. Usually we need to use new ideas in constructing bijective proofs, or prove the equivalence recursively using recurrence relations.
\end{enumerate}
\end{Observation}

On the other hand, the statistic $\comp$ behaves better under these three elementary operations. 
\begin{Observation}\label{obs:2}
The two mappings $\pi\mapsto (\pi^r)^c$ and $\pi\mapsto \pi^{-1}$ both preserve the statistic $\comp$. 
\end{Observation}

Let $P$ be a collection of patterns. The following trivariate generating function will be the focal point of our study.
\begin{align}\label{gf:main}
\S(P)^{\des,\iar,\comp}(t, r, p; z):=\sum_{n\ge 0}\sum_{\pi\in\S_n(P)}t^{\des(\pi)}r^{\iar(\pi)}p^{\comp(\pi)}z^n.
\end{align}
Most of the time we suppress the superindices $\des,\iar,\comp$, and variable $z$, and when the pattern set $P$ is clear from the context, we also suppress $P$ to write $\S(t,r,p)$. In most cases, we simply calculate the variant $\tilde\S(t,r,p):=(\S(t,r,p)-1)/rpz$, so that the final expressions of the generating functions are more compact to be collected and displayed in a table (see Tables~\ref{one 3-avoider} and \ref{two 3-avoiders}). Let $M_n(P)$ be the $n\times n$ matrix, whose entry at the $k$-th row and the $\ell$-th column is the number of permutations $\pi$ in $\S_n(P)$ with $\iar(\pi)=k$ and $\comp(\pi)=\ell$. Let $\st$ be a permutation statistic, we can then refine $M_n(P)$ as $M_n(P)=\sum_{i}M_n^{\st=i}(P)$, so that the $(k,\ell)$-entry of $M_n^{\st=i}(P)$ counts permutations $\pi$ such that $\st(\pi)=i$ for certain fixed integer $i$. This definition extends to set-valued statistics and multiple statistics in a natural way. So for instance, $M_n^{\LMAX=S,\des=i}(P)$ is the $n\times n$ matrix, whose $(k,\ell)$-entry is the number of permutations $\pi$ in $\S_n(P)$ with $\LMAX(\pi)=S$, $\des(\pi)=i$, $\iar(\pi)=k$ and $\comp(\pi)=\ell$. 

We also need the following operations on permutations.

\begin{Def}\label{del and ins}
For a word $w$ over $\Z$, denote $red(w)$ the {\em reduction} of $w$, which is obtained from $w$ by replacing the $j$-th smallest positive letter by $j$. For a given permutation $\pi\in\S_n$, the {\em deletion of $i$}, for each $i\in[n]$, is the map that deletes $i$ from $\pi$, and reduces the derived word to a permutation, denoted as $\del_i(\pi)\in\S_{n-1}$. Similarly, the {\em insertion of $i$ at place $k$}, for each $i,k\in [n+1]$, is defined to be the map that increases all letters $j\ge i$ in $\pi$ by $1$, and inserts $i$ between $\pi(k-1)$ and $\pi(k)$ to get a new permutation, denoted as $\ins_{i,k}(\pi)\in\S_{n+1}$.
\end{Def}

\subsection{The direct/skew sum  operation and fundamental lemmas}
\label{sec:dirsum}

There are two fundamental operations, called direct sum and skew sum, to construct a bigger permutation from two smaller ones. 
The \emph{direct sum} $\pi\oplus\sigma$ and the \emph{skew sum} $\pi\ominus\sigma$, of $\pi\in\S_k$ and $\sigma\in\S_l$, are permutations in $\S_{k+l}$ defined respectively as
$$
(\pi\oplus\sigma)_i=
\begin{cases}
\pi_i, &\text{for $i\in[1,k]$};\\
\sigma_{i-k}+k, &\text{for $i\in[k+1,k+l]$}
\end{cases}
$$
and
$$
(\pi\ominus\sigma)_i=
\begin{cases}
\pi_i+l, &\text{for $i\in[1,k]$};\\
\sigma_{i-k}, &\text{for $i\in[k+1,k+l]$}.
\end{cases}
$$
 For instance, we have $123\oplus 21=12354$ and $123\ominus 21=34521$. 
 The following  characterization of separable permutations is  folkloric  (see \cite[pp.~57]{kit}) in pattern avoidance. 
 \begin{proposition}\label{desides}
 A permutation is separable if and only if it can be built from the permutation $1$ by applying the operations $\oplus$ and $\ominus$ repeatedly. 
 \end{proposition}
 
A nonempty permutation which is not the direct sum of two nonempty permutations is called {\em indecomposable}. Any permutation $\pi$ with $\comp(\pi)=k$ can be written uniquely as $\pi=\tau_1\oplus\tau_2\oplus\cdots\oplus\tau_k$, where each $\tau_i$ is indecomposable. We call such decomposition the {\em direct sum decomposition} of $\pi$.
Let $\id_n$ denote the identity permutation of length $n$. A statistic $\st$ is called  {\em totally  $\oplus$-compatible}  if  
$
\st(\pi)=\sum_{i=1}^k\st(\tau_i)
$
and is called {\em partially $\oplus$-compatible} if 
$
\st(\pi)=\sum_{i=1}^l\st(\tau_i)$,
where $l=\min(\{i:\tau_i\neq \id_1\}\cup\{k\})$. For instance, $\des$ and $\comp$ are totally  $\oplus$-compatible, while $\iar$ is partially $\oplus$-compatible. We emphasize here that totally $\oplus$-compatibility does not imply partially $\oplus$-compatibility.

Let $P$ be a collection of patterns and  $(\st_1,\st_2,\ldots)$ be a sequence of permutation statistics. Let us introduce two generating functions with respect to $(\st_1,\st_2,\ldots)$ as
$$
F_P(t_1,t_2,\ldots;z):=1+\sum_{n\geq1}z^n\sum_{\pi\in\S_n(P)}\prod_it_i^{\st_i(\pi)}
$$
and 
$$
I_P(t_1,t_2,\ldots;z):=\sum_{n\geq1}z^n\sum_{\pi\in\In_n(P)}\prod_it_i^{\st_i(\pi)},
$$
where $\In_n$ denotes the set of all indecomposable permutations of length $n$. 
We have the following general lemma regarding the direct sum decomposition of permutations, which is useful when considering the refinement of Wilf-equivalence by $\comp$. 

\begin{lemma} \label{general}
Let $(\st_1,\st_2,\ldots,\st'_1,\st'_2,\ldots)$ be a sequence of  statistics such that $\st_i$ is totally  $\oplus$-compatible  and $\st'_i$ is partially  $\oplus$-compatible for each $i$.
Let $P$ and $Q$ be two collections of indecomposable patterns. We claim
\begin{enumerate}
	\item We have the following functional equation:
	\begin{equation}\label{relation:comp}
	F_P(q)=\frac{1}{1-qw}+\frac{q(I_P({\bf t}, {\bf t'})-w)}{(1-qI_P({\bf t},{\bf 1}))(1-qw)},
	\end{equation}
	where $w=zt_1^{\st_1(\id_1)}t_2^{\st_2(\id_1)}\cdots {t'_1}^{\st'_1(\id_1)}{t'_2}^{\st'_2(\id_1)}\cdots$, and \begin{align*}F_P(q):=F_P(q,t_1,t_2,\ldots,t'_1,t'_2,\ldots;z) \text{ and } I_P({\bf t}, {\bf t'}):=I_P(t_1,t_2,\ldots, t'_1,t'_2,\ldots;z)
	\end{align*} 
	are the generating functions with respect to $(\comp,\st_1,\st_2,\ldots,\st'_1,\st'_2\ldots)$ and \linebreak $(\st_1,\st_2,\ldots,\st'_1,\st'_2\ldots)$, respectively. In particular, $I_P({\bf t},{\bf 1}):=I_P(t_1,\ldots, 1,\ldots;z)$.
	\item If $P\sim_{(\st_1,\st_2,\ldots,\st'_1,\st'_2,\ldots)}Q$, then $P\sim_{(\comp,\st_1,\st_2,\ldots,\st'_1,\st'_2,\ldots)}Q$ holds as well. In particular, if $P\sim Q$, then $P\sim_{\comp} Q$.
\end{enumerate}
\end{lemma}
\begin{proof}
Note that if $\sigma$ is an indecomposable pattern and $\pi=\tau_1\oplus\tau_2\oplus\cdots\oplus\tau_k$,  then
$$
\pi \text{ is $\sigma$-avoiding}\Longleftrightarrow \tau_i\text{ is $\sigma$-avoiding for each $i$}. 
$$
By general principles, the weight of $\pi$ that contributes to the generating function $F_P(q)$ is the product of the weights of $\tau_1,\tau_2,\ldots,\tau_k$. Among these $k$ indecomposable components, suppose the first $i$ are trivial (i.e., $\id_1$) with weight $w$, the $(i+1)$-th component is nontrivial thus generated by $I_P({\bf t},{\bf t'})-w$, and the remaining $k-i-1$ components do not affect those partially $\oplus$-compatible statistics ${\bf t'}$, thus each is generated by $I_P({\bf t},{\bf 1})$. The discussions above amount to give us 
\begin{align*}
F_P(q)&=1+\sum_{n\geq1} q^n(w^n+\sum_{i=0}^{n-1}w^i(I_P({\bf t}, {\bf t'})-w)I_P({\bf t},{\bf 1})^{n-1-i})\\
&=\frac{1}{1-qw}+\frac{I_P({\bf t}, {\bf t'})-w}{I_P({\bf t},{\bf 1})-w}\biggl(\frac{qI_P({\bf t},{\bf 1})}{1-qI_P({\bf t},{\bf 1})}-\frac{qw}{1-qw}\biggr),
\end{align*}
which becomes~\eqref{relation:comp}  after simplification.

In view of \eqref{relation:comp}, the following three statements are equivalent: 
\begin{itemize}
\item [(i)] $F_P(1)=F_Q(1)$, namely $P\sim_{(\st_1,\st_2,\ldots,\st'_1,\st'_2,\ldots)}Q$.
\item [(ii)] $I_P({\bf t}, {\bf t'})=I_Q({\bf t}, {\bf t'})$. 
\item [(iii)] $F_P(q)=F_Q(q)$. 
\end{itemize}
Thus, statement (i) is equivalent to its seemingly stronger form  (iii), as desired. 
\end{proof}


The following general lemma indicates that for a collection of indecomposable patterns, say $P$, the equidistribution of certain statistic $\st$ with $\comp$ over $\S_n(P)$, implies the seemingly stronger result that the joint distribution $(\st, \comp)$ is symmetric over $\S_n(P)$. This result is somewhat surprising. 
\begin{lemma}\label{general:sym}
Let $P$ be a collection of indecomposable patterns. Let $\st$ be a partially  $\oplus$-compatible statistic such that $\st(\id_1)=1$ and  $(\st_1,\st_2,\ldots)$ be a sequence of totally  $\oplus$-compatible statistics. If $|\S_n(P)^{\st,\st_1,\st_2,\ldots}|=|\S_n(P)^{\comp,\st_1,\st_2,\ldots}|$, then 
$$
|\S_n(P)^{\st,\comp,\st_1,\st_2,\ldots}|=|\S_n(P)^{\comp,\st,\st_1,\st_2,\ldots}|.
$$
In particular, if $\st$ is a Comtet statistic over $\S_n(P)$, then $(\st,\comp)$ is a symmetric pair of Comtet statistics over $\S_n(P)$. 
\end{lemma}
\begin{proof}
Let $F_P(r,s):=F_P(r,s,t_1,t_2,\ldots;z)$  and $I_P(s):=I_P(s,t_1,t_2,\ldots;z)$ be the generating functions with respect to $(\comp,\st,\st_1,\st_2,\ldots)$ and $(\st,\st_1,\st_2,\ldots)$, respectively.  By the realationship~\eqref{relation:comp}, we have 
\begin{equation}\label{eq1:comp}
F_P(r,s)=\frac{1}{1-rsz}+\frac{r(I_P(s)-sz)}{(1-rI_P(1))(1-rsz)}.
\end{equation}
Since $F_P(1,s)=F_P(s,1)$, it follows from the above identity that 
$$
\frac{1}{1-sz}+\frac{I_P(s)-sz}{(1-I_P(1))(1-sz)}=\frac{1}{1-sz}+\frac{sI_P(1)-z}{(1-sI_P(1))(1-z)}.
$$
Solving this equation gives 
$$
I_P(s)=\frac{sI_P(1)(sz-z-1+I_P(1))}{1-sI_P(1)}.
$$
Plugging this into~\eqref{eq1:comp} results in 
\begin{equation}\label{gen:symm}
F_P(r,s)=\frac{1-rsz+(rsz+rs-r-s)I_P(1)}{(1-rI_P(1))(1-sI_P(1))(1-rsz)},
\end{equation}
which is symmetric in $r$ and $s$. This completes the proof of the lemma.
\end{proof}



\section{A single pattern of length 3}\label{sec:one 3-pattern}

In this section, we deal with all patterns $\tau$ of length $3$ and complete two tasks: 
\begin{itemize}
	\item[1)] Show symmetry of the Comtet pair $(\iar,\comp)$, jointly with some other (set-valued) statistics, over certain class of pattern-avoiding permutations or admissible words (see Theorem~\ref{thm:AW-Hankel}). In all cases the proofs are combinatorial. We collect all the bijections here for easy reference: $\xi$ (Theorem~\ref{thm:321=312}), $\alpha$ and $\beta$ (Theorem~\ref{thm:alpha-beta}), $\psi$ (Theorem~\ref{thm:AW-Hankel}), $\varphi$ (Theorem~\ref{thm:132-symm}), and $\theta$ (Theorem~\ref{thm:213-231-conjugate}).
 	\item[2)] Compute the trivariate generating function $\S(\tau)^{\des,\iar,\comp}(t,r,p)$, which leads to full $\iar$- and $\comp$-Wilf-equivalence classification. A snapshot of these results is presented in Table~\ref{one 3-avoider}. Putting $t=1$, and $p=1$ (or $r=1$) in the generating functions listed in Table~\ref{one 3-avoider} and comparing the results, we can conclude that there are three $\iar$-Wilf-equivalence classes:
	$$\{213,312,321\},\; \{132,231\}, \; \text{and }\{123\}.$$ 
	While the $\comp$-Wilf-equivalence classes are:
	$$\{231,312,321\},\; \{132,213\}, \; \text{and }\{123\}.$$
\end{itemize}

\begin{table}
{\small
\begin{tabular}{cccc}
\toprule
$P$\,\, & $\tilde\S(t,r,p)$ & $M_n(P)$ & proved in\\
\midrule
&&&\\
$312$ & $\dfrac{1-(r+p+tN)z+(rp+(r+p-1)tN)z^2}{(1-rpz)(1-rz-tNz)(1-pz-tNz)}$ & Symmetric & Thm.~\ref{gf:312}\\
&&&\\
$321$ & $\dfrac{(rpz-rz+tz)C^2-(rpz+p-1)C+p}{(1-rpz)(1-rzC)(p+C-pC)}$ & Equals $M_n(312)$  & Thm.~\ref{gf:321}\\
&&&\\
 
$132$ &\,\, $\dfrac{1}{1-rpz}+\dfrac{(1-z)(N-1)t}{(1-rz)(1-pz)(1-z-(N-1)tz)}$ & Hankel & Thm.~\ref{thm:gf-132}\\
&&&\\
$213$ & $\dfrac{(1-rz)(tN-t+1)}{(1-rpz)(1-rz(tN-t+1))}$ & Lower triangular & Thm.~\ref{thm:gf-213a231}\\
&&&\\
$231$ & $\dfrac{(1-pz)(tN-t+1)}{(1-rpz)(1-pz(tN-t+1))}$ & Conjugate to $M_n(213)$ & Thm.~\ref{thm:gf-213a231}\\
&&&\\
$123$ & $\dfrac{(1-p)z(trz-tz-r)}{(1-tz)^2}+\dfrac{(1+rz-tz) C^*}{z(1+z-tz)}$ & $2\times 2$ nonzero & Thm.~\ref{thm:gf-123}\\
&&&\\
\bottomrule\\
\end{tabular}
}
\caption{One pattern of length $3$\label{one 3-avoider} (definitions of $N$, $C$ and $C^*$ are given in equations \eqref{def:N}, \eqref{eq:Barn} and \eqref{def:C^*}, respectively)}
\end{table}

\subsection{Symmetric classes}\label{sym:3.1}
For the three patterns $312$, $321$ and $132$, the distributions of $\iar$ and $\comp$ are not only identical, but also jointly symmetric. For the two indecomposable patterns $312$ and $321$, this stronger property can be deduced from Lemma~\ref{general:sym}. But for the pattern $132=1\oplus 21$, we need to construct an involution $\varphi$ on $\S_n(132)$, which actually enables us to derive a more refined equidistribution (see Theorem~\ref{thm:132-symm}). We begin with the patterns the $321$ and $312$.


\medskip
\noindent \underline{\bf Patterns $312$ and $321$}
\medskip

Pattern $321$ seems to always attract more attention than the rest of patterns in $\S_3$, perhaps because of its role in Deodhar's combinatorial framework for determining the Kazhdan-Lusztig polynomials (see for instance \cite{BW}). Rubey \cite{rub} obtained an equidistribution result over $\S_n(321)$ by first mapping each $321$-avoiding permutation, along with the statistics involved, to a Dyck path via Krattenthaler's bijection \cite{kra}, and then constructing an involution on Dyck paths. We restate his result here using $321$-avoiding permutations rather than Dyck paths. For each $\pi\in\S_n$, let 
$$\ldes(\pi):=\max(\{0\}\cup\DES(\pi))
$$
be the position of the {\em {\bf\em l}ast {\bf\em des}cent} of $\pi$. Recall the boldface notation ${\bf x}^S$ defined in Theorem~\ref{thm:sep:sym}.
\begin{theorem}[Rubey~\cite{rub}]\label{thm:Rubey}
There exists an involution on $\S_n(321)$ which proves the equidistribution
\begin{equation}\label{equ:rubey}
\sum_{\pi\in\S_n(321)}s^{\comp(\pi)}t^{n-\ldes(\pi^{-1})}{\bf x}^{\LMAXP(\pi)}=\sum_{\pi\in\S_n(321)}s^{n-\ldes(\pi^{-1})}t^{\comp(\pi)}{\bf x}^{\LMAXP(\pi)}.
\end{equation}
\end{theorem} 

We explain here why Theorem~\ref{thm:Rubey} is equivalent to our Theorem~\ref{main:thm1} (i) up to the elementary transformation $\pi\mapsto (\pi^{-1})^{\mathrm{rc}}$. Notice that for each $\pi\in\S_n$, we have the relationships 
\begin{align*}
&n-\ldes(\pi^{\mathrm{rc}})=\iar(\pi)\quad\text{and}\\
&\LMAXP(\pi)=\LMIN(\pi^{-1})=\overline{\LMAX((\pi^{-1})^{\mathrm{rc}})},
\end{align*}
where $\bar{S}:=\{n+1-i: i\in S\}$ for any subset $S\subseteq[n]$. In view of these relationships and Observation~\ref{obs:2}, we have 
\begin{align*}
\sum_{\pi\in\S_n(321)}s^{\comp(\pi)}t^{n-\ldes(\pi^{-1})}{\bf x}^{\LMAXP(\pi)}&=\sum_{(\pi^{-1})^{\mathrm{rc}}\in\S_n(321)}s^{\comp((\pi^{-1})^{\mathrm{rc}})}t^{n-\ldes(\pi^{\mathrm{rc}})}{\bf x}^{\LMAXP((\pi^{-1})^{\mathrm{rc}})}\\
&=\sum_{(\pi^{-1})^{\mathrm{rc}}\in\S_n(321)}s^{\comp(\pi)}t^{\iar(\pi)}{\bf x}^{\overline{\LMAX(\pi)}}\\
&=\sum_{\pi\in\S_n(321)}s^{\comp(\pi)}t^{\iar(\pi)}{\bf x}^{\overline{\LMAX(\pi)}}.
\end{align*}
Therefore, equidistribution~\eqref{equ:rubey} is equivalent to Theorem~\ref{main:thm1} (i). 

In view of Lemma~\ref{general} (2), $321\sim_{\comp}312$ since $321\sim 312$. We have the following refinement.
\begin{theorem}\label{thm:321=312}
For each $n\ge 1$, there exists a bijection $\xi$, mapping each $\pi\in\S_n(321)$ onto $\sigma:=\xi(\pi)\in\S_n(312)$, such that
\begin{align}\label{eq:321=312}
(\LMAX,\LMAXP,\iar,\comp)\:\pi = (\LMAX,\LMAXP,\iar,\comp)\:\sigma.
\end{align}
\end{theorem}

Sitting in the heart of our proof of Theorem~\ref{thm:321=312}, is a certain word composed of positive integers and a symbol $\diamond$ that stands for an empty slot, which we introduce now.
\begin{Def}
Given a nonempty set $S=\{s_1,\ldots,s_k\}\subseteq \Z_{>0}$ with $s_1<\cdots<s_k$, and a weak composition $c=(c_1,\ldots,c_k)$ of $s_k-k$, we form a word $$w_{S,c}:=s_1\underbrace{\diamond\cdots\diamond}_{c_1}s_2\underbrace{\diamond\cdots\diamond}_{c_2}s_3\cdots s_k\underbrace{\diamond\cdots\diamond}_{c_k}.$$ It is said to be an {\em admissible word with respect to} $S$ and $c$, if for $1\le i \le k$,
\begin{align}
\sum_{j=1}^{i}c_j\le s_i-i. \tag{$\ast$}\label{word condition}
\end{align}
Let $\AW_n$ denote the set of all admissible words of length $n$.
\end{Def}
We also need to introduce the counterparts on $\AW_n$ of the quadruple statistics in \eqref{eq:321=312}. For each $w:=w_{S,c}\in\AW_n$, let $\ics(w)$ denote the number of {\em {\bf\em i}nitial {\bf\em c}onsecutive letters from $S$} in $w$, $\equ(w)$ denote the number of times the condition \eqref{word condition} is satisfied with an {\em {\bf\em equ}al} sign, and $\SP(w)$ denote the set of positions (in $w$) of letters from $S$. For example, if $w=2\:3\:5\diamond7\diamond \diamond 10\:12\diamond 13\diamond\diamond$ with $S=\{2,3,5,7,10,12,13\}$, then $\ics(w)=3$, $\equ(w)=2$, $\SP(w)=\{1,2,3,5,8,9,11\}$.

\begin{theorem}\label{thm:alpha-beta}
There exist two bijections $\alpha:\S_n(321)\rightarrow \AW_n$ and $\beta:\S_n(312)\rightarrow \AW_n$, such that for any $\pi\in\S_n(321)$ and $\sigma\in\S_n(312)$, we have
\begin{align}
(\LMAX,\LMAXP,\iar,\comp)\:\pi = (S,\SP,\ics,\equ)\:w_{S,c},\label{eq:alpha} \\
(\LMAX,\LMAXP,\iar,\comp)\:\sigma = (T,\SP,\ics,\equ)\:w_{T,d},\label{eq:beta}
\end{align}
where $w_{S,c}=\alpha(\pi)$ and $w_{T,d}=\beta(\sigma)$.
\end{theorem}
\begin{proof}
Since the constructions for the two bijections $\alpha$ and $\beta$ are almost the same (the only difference lies in their inverses), we will give details mainly for $\alpha$. For each $\pi\in\S_n(321)$, suppose
\begin{align*}
&S:=\LMAX(\pi)=\{\pi(i_1)=\pi(1),\pi(i_2),\ldots,\pi(i_k)\}.
\end{align*} 
Let $c=(c_1,\ldots,c_k)$, with $c_h=i_{h+1}-i_h-1$, for $1\le h\le k-1$, $c_k=n-i_k$. In other words, each part of the composition $c$ records the number of letters between two left-to-right maxima, after having appended $n+1$ to the permutation $\pi$. Now we define $\alpha(\pi):=w_{S,c}.$ Note that $\pi(i_1),\ldots,\pi(i_k)$ are the left-to-right maxima of $\pi$, so we can verify the condition \eqref{word condition} holds for $S$ and $c$, therefore $\alpha$ is a well-defined map from $\S_n(321)$ to $\AW_n$. The map $\beta$ is defined analogously, only that now the preimage is a $312$-avoiding, rather than $321$-avoiding permutation. Now we show both $\alpha$ and $\beta$ are bijections by constructing their inverses. Take a word $w_{S,c}\in\AW_n$, we replace all the $\diamond$'s from left to right with the smallest unused letter in $[n]\setminus S$. This results in a $321$-avoiding permutation, say $\hat{\pi}$. On the other hand, if we replace all the $\diamond$'s from left to right with the largest unused letter in $[n]\setminus S$, keeping letters from $S$ the left-to-right maxima, we will end up with a $312$-avoiding permutation, say $\hat{\sigma}$. 



It should be clear that
\begin{align*}
&\LMAX(\hat\pi)=S=\LMAX(\hat\sigma),\\
&\LMAXP(\hat\pi)=\SP(w_{S,c})=\LMAXP(\hat\sigma),\\
&\iar(\hat\pi)=\ics(w_{S,c})=\iar(\hat\sigma),\\
&\comp(\hat\pi)=\equ(w_{S,c})=\comp(\hat\sigma).
\end{align*}
Now set $\alpha^{-1}(w_{S,c})=\hat\pi$ (resp.~$\beta^{-1}(w_{S,c})=\hat\sigma$). Evidently, $$\alpha^{-1}(\alpha(\pi))=\pi,\quad \beta^{-1}(\beta(\sigma))=\sigma,$$
so $\alpha$ and $\beta$ are indeed bijections that transform the quadruple statistics as shown in \eqref{eq:alpha} and \eqref{eq:beta}.
\end{proof}

\begin{proof}[{\bf Proof of Theorem~\ref{thm:321=312}}]
Simply set $\xi=\beta^{-1}\circ \alpha$, and \eqref{eq:321=312} follows immediately from \eqref{eq:alpha} and \eqref{eq:beta}.
\end{proof}

\begin{remark}
When composed with the complement map, our bijection $\xi$ is equivalent to Simion and Schmidt's \cite{SS} bijection from $\S_n(123)$ to $\S_n(132)$. This bijection is also called {\em the Knuth--Richards bijection} by Claesson and Kitaev \cite{CK}, see also \cite{doy}.
\end{remark}

In view of \eqref{eq:alpha}, the pair $(\ics,\equ)$ on admissible words corresponds to the pair $(\iar,\comp)$ on $321$-avoiding permutations, so Rubey's Theorem~\ref{thm:Rubey} tells us that their distributions are jointly symmetric over $\AW_n$. Note that Rubey's proof was via an involution on Dyck paths. We are able to construct an invertible map $\psi$ over the set of admissible words. To facilitate the description of $\psi$, we need the following definition.
\begin{Def}\label{critical}
Given an admissible word $w_{S,c}$ with $S=\{s_1,\ldots,s_k\}$ and $c=(c_1,\ldots,c_k)$, the index $i$, $1\le i< k$ is said to be {\em critical} for $w$, if $$\sum_{j=1}^{i}c_j< s_i-i\le \sum_{j=1}^{i+1}c_j.$$
\end{Def}

For the previous example $w=2\:3\:5\diamond7\diamond \diamond 10\:12\diamond 13\diamond\diamond$, we see the indices $2,3$ and $6$ are critical for $w$. Let $\AW_{n,a,b}$ denote the set of admissible words $w:=w_{S,c}\in\AW_n$ such that $\ics(w)=a$, $\equ(w)=b$ and $s_1>1$, where $s_1$ is the smallest letter in $S$. 

\begin{theorem}\label{thm:AW-Hankel}
For $1< a\le n$ and $1\le b< n$, there exists a bijection $\psi$ from $\AW_{n,a,b}$ to $\AW_{n,a-1,b+1}$, such that for each $w_{S,c}\in\AW_{n,a,b}$, if $\psi(w_{S,c})=v_{T,d}$, then we have $S=T$.
\end{theorem}
\begin{proof}
Take any $w:=w_{S,c}\in\AW_{n,a,b}$ with $S=\{s_1,\ldots,s_k\}$ and $c=(c_1,\ldots,c_k)$, we explain how to produce an admissible word $v:=v_{S,d}$ such that $\ics(v)=\ics(w)-1$ and $\equ(v)=\equ(w)+1$. Since $\ics(w)=a\ge 2$, we see $c_1=c_2=\cdots=c_{a-1}=0$ and $c_a>0$. Find the smallest $\ell\ge a-1$ such that the index $\ell$ is critical for $w$. Note that $s_1>1$ guarantees the existence of such an $\ell$. Let $d=(d_1,\cdots,d_k)$ be defined as
$$d_i=\begin{cases}
c_{i+1} & \text{if } a-1\le i\le \ell-1,\\
s_i-i-\sum_{h=1}^{i}c_h & \text{if } i=\ell,\\
\sum_{h=1}^i c_h-\sum_{h=1}^{i-1} d_h & \text{if } i=\ell+1,\\
c_i & \text{otherwise.}\\
\end{cases}$$
We denote $v:=v_{S,d}$ the admissible word with respect to $S$ and $d$, and set $\psi(w)=v$. It can be checked that $\sum_{i=1}^k c_i=\sum_{i=1}^k d_i=s_k-k$ and $\sum_{i=1}^{\ell} d_i=s_{\ell}-\ell$, hence $\equ(v)=\equ(w)+1$ as desired. Also $\ics(v)=\ics(w)-1=a-1$ since now $d_1=\cdots=d_{a-2}=0$ and $d_{a-1}=c_a>0$. 

All it remains is to show that $\psi$ is invertible. To this end, for each $v:=v_{S,d}\in\AW_{n,a-1,b+1}$, find the smallest integer $\ell$ such that $\sum_{i=1}^{\ell} d_i=s_{\ell}-\ell$. Note that since $\equ(v)=b+1\ge 2$, $s_1>1$ and $d_1=\cdots=d_{a-2}=0$, we must have $a-1\le \ell <k$, and $\ell$ being the smallest means $d_{\ell}>0$. Now let $c=(c_1,\ldots,c_k)$ be defined as
$$c_i=\begin{cases}
d_{i-1} & \text{if } a\le i\le \ell,\\
0 & \text{if } i = a-1,\\
d_{i-1}+d_i & \text{if } i=\ell+1,\\
d_i & \text{otherwise.}\\
\end{cases}$$
It is routine to check that $w:=w_{S,c}$ is the desired preimage so that $\psi(w)=v$, $\ics(w)=\ics(v)+1$, and $\equ(w)=\equ(v)-1$. 
\end{proof}


The following result is the restatement of Theorem~\ref{main:thm1} (i) and (ii).
\begin{corollary}\label{cor:321-312-symmetry}
For every $n\ge 1$, the two triples $(\LMAX,\iar,\comp)$ and $(\LMAX,\comp,\iar)$ have the same distribution over $\S_n(321)$; the two quadruples $(\LMAX,\DESB,\iar,\comp)$ and $(\LMAX,\DESB,\comp,\iar)$ have the same distribution over $\S_n(312)$.
\end{corollary}

\begin{proof}
For each permutation $\pi\in\S_n(321)$ with $\pi(1)>1$, we find a unique permutation $\rho\in\S_n(321)$ such that $$(\LMAX,\iar,\comp)\:\pi = (\LMAX,\comp,\iar)\:\rho.$$ If $\iar(\pi)=\comp(\pi)$, then simply take $\rho=\pi$. Otherwise we assume $\iar(\pi)=\comp(\pi)+k$ for some $k\neq 0$, let $$\rho=\alpha^{-1}(\psi^k(\alpha(\pi))).$$
Combining Theorem~\ref{thm:AW-Hankel} with \eqref{eq:alpha}, we verify that
\begin{align*}
&\LMAX(\rho)=\LMAX(\pi),\\
&\iar(\rho)=\ics(\psi^k(\alpha(\pi)))=\ics(\alpha(\pi))-k=\iar(\pi)-k=\comp(\pi), \; \text{and}\\
&\comp(\rho)=\equ(\psi^k(\alpha(\pi)))=\equ(\alpha(\pi))+k=\comp(\pi)+k=\iar(\pi), 
\end{align*}
as desired. Now both $\alpha$ and $\psi$ are bijections, so $\pi$ and $\rho$ are in one-to-one correspondence. On the other hand, for each $\pi\in\S_n(321)$ with $\pi(1)=1$, we see $\nu:=\del_1(\pi)\in\S_{n-1}(321)$ satisfies $\iar(\nu)=\iar(\pi)-1$, $\comp(\nu)=\comp(\pi)-1$, and $\LMAX(\nu)$ is the set obtained from decreasing each number in $\LMAX(\pi)\setminus\{1\}$ by $1$. This means we can use induction to finish the proof of the result for $\S_n(321)$.

Finally, applying the bijection $\beta$ instead of $\alpha$ gives us the result for $\S_n(312)$. To see why we can include $\DESB$ to have a quadruple in this case, simply observe that for each permutation $\sigma\in\S_n(312)$, $\LMAX(\sigma)\cup\DESB(\sigma)=[n]$.
\end{proof}

For most of our calculations of the generating function $\S(P)(t,r,p)$ in this and later sections, we use some kind of decomposition by considering the largest (resp.~smallest) letter $n$ (resp.~$1$) in a permutation $\sigma\in\S_n$. A maximal consecutive subset of $[n]$, all of whose elements appear on the same side of $n$  (resp.~$1$) in $\sigma$, is called a {\em block with respect to $n$ (resp.~$1$)}. For example, the blocks with respect to $9$ in $251986743$ are $\{1,2\}$, $\{3, 4\}$, $\{5\}$ and $\{6, 7, 8\}$. For two blocks (or sets) $A$ and $B$, we write $A<B$ if the maximal element of $A$ is smaller than the minimal element of $B$. As usual, we use $\chi(\mathsf{S})=1$ if the statement $\mathsf{S}$ is true, and $\chi(\mathsf{S})=0$ otherwise.

A square matrix is said to be {\em Hankel} if it has constant skew-diagonals. For the next theorem and Theorems \ref{thm:gf-132} and \ref{thm:gf-132-312:132-321}, a key fact utilized by us is that $M_n(P)$ or $M_n^{\st=i}(P)$ is a Hankel matrix. This not only implies that $(\iar,\comp)$ is symmetric over $\S_n(P)$, but also facilitates our calculation of the generating function $\S(P)^{\des,\iar,\comp}(t,r,p;z)$. We elaborate on the latter point with the next lemma. 


\begin{lemma}\label{gf:Hankel}
Suppose $M=(m_{ij})_{1\le i,j\le n}$ is a Hankel matrix such that $m_{ij}=0$ when $i+j\ge n+2$. Let $\mathcal{M}(x,y):=\sum_{1\le i,j\le n}m_{ij}x^iy^j$ and $\mathcal{N}(x):=\frac{\partial \mathcal{M}}{\partial y}|_{y=0}=\sum_{1\le i\le n}m_{i1}x^i$ be the generating functions of $M$ and its first column, respectively. It holds that
\begin{align}
\mathcal{M}(x,y)=\frac{xy}{x-y}(\mathcal{N}(x)-\mathcal{N}(y)).
\end{align}
\end{lemma}
\begin{proof}
The Hankel condition enables us to group together terms along the same skew-diagonal. Noting that $x^iy+x^{i-1}y^2+\cdots+xy^i=xy(x^i-y^i)/(x-y)$ for each $1\le i\le n$, we have
\begin{align*}
\mathcal{M}(x,y) &= \sum_{i=1}^n (m_{i1}x^iy+m_{i-1\:2}x^{i-2}y^2+\cdots+m_{1i}xy^i)=\sum_{i=1}^n m_{i1}(x^iy+\cdots+xy^i)\\
&= \frac{xy}{x-y}\sum_{i=1}^n m_{i1}(x^i-y^i)=\frac{xy}{x-y}(\mathcal{N}(x)-\mathcal{N}(y)),
\end{align*}
as desired.
\end{proof}

Recall the {\em Narayana polynomial} $N_n(t):=\sum_{\pi\in\S_n(\tau)}t^{\des(\pi)}$ ($\tau=312, 213, 132$ or $231$) and its generating function (see e.g.~\cite[Eq.~2.6]{pet})
\begin{align}\label{def:N}
N:=N(t;z):=\sum_{n\ge 0}N_n(t)z^n=\frac{1+(t-1)z-\sqrt{1-2(t+1)z+(t-1)^2z^2}}{2tz}.
\end{align}
 
\begin{theorem}\label{gf:312}
The generating function of the triple statistic $(\des,\iar,\comp)$ over $\S_n(312)$ is given by
\begin{align}\label{eq:gf-312}
\tilde \S(312)^{\des,\iar,\comp}(t,r,p)=\frac{1-(r+p+tN)z+(rp+(r+p-1)tN)z^2}{(1-rpz)(1-rz-tNz)(1-pz-tNz)}.
\end{align}
\end{theorem}
\begin{proof}
Conditioning on the first letter $\pi(1)$, we claim that (the pattern ``$312$'' have all been suppressed for brevity)
\begin{align}\label{eq:Phi-312}
\S(t,r,p) &= 1+rpz\S(t,r,p)+\frac{rp}{r-p}(\tilde I(t,r)-\tilde I(t,p)),
\end{align}
where
\begin{align}
\label{tilde I}\tilde I(t,r) &:= \sum_{n\ge 1}z^n\sum_{\substack{\pi\in\S_n(312) \\ \pi(1)>1,\:\comp(\pi)=1}}t^{\des(\pi)}r^{\iar(\pi)}=\left.\frac{\S(t,r,p)-1-rpz\S(t,r,p)}{p}\right|_{p=0}\\
&= rz(\tilde \S(t,r,0)-1).\nonumber
\end{align}
Indeed, the first summand $1$ in \eqref{eq:Phi-312} corresponds to the empty permutation, and the second to those with $\pi(1)=1$. As for the third summand, we consider permutations $\pi$ with $\pi(1)>1$. Now Eq.~\eqref{eq:beta} and Theorem~\ref{thm:AW-Hankel} tell us that for a given $1\not\in S\subseteq [n]$, the matrix $M_n^{\LMAX=S}(312)$ is Hankel. Moreover, Lemma~\ref{gf:Hankel} is applicable since the only permutation with $\iar(\pi)=n$ is $\pi=\id_n$ but we require that $\pi(1)>1$. Lastly, as we have already noted in the proof of Corollary~\ref{cor:321-312-symmetry}, each permutation $\sigma\in\S_n(312)$ satisfies $\LMAX(\sigma)\cup\DESB(\sigma)=[n]$. This means in particular that the statistic $\des$ takes the same value for all permutations enumerated by $M_n^{\LMAX=S}(312)$, justifying the variable $t$ in \eqref{tilde I}.


Next, plugging \eqref{tilde I} into \eqref{eq:Phi-312} yields
\begin{align}\label{eq:tildePhi}
(r-p)(1-rpz)\tilde \S(t,r,p)=r\tilde \S(t,r,0)-p\tilde \S(t,p,0).
\end{align}
Setting $p=1$ in \eqref{eq:tildePhi}, solving for $\tilde\S(t,r,0)$ and then plugging back into \eqref{eq:tildePhi} gives us
\begin{align}
& (r-p)(1-rpz)\tilde\S(t,r,p) = (r-1)(1-rz)\tilde\S(t,r,1)-(p-1)(1-pz)\tilde\S(t,p,1).\label{eq:Phi-r-p}
\end{align}
It remains to calculate $\tilde\S(t,r,1)$. Every nonempty $312$-avoiding permutation $\pi$ has the block decomposition $\pi=A\:1\:B$ such that $A$ and $B$ are both $312$-avoiding blocks with $A<B$. We consider the following two cases:
\begin{itemize}
\item $A=\emptyset$, i.e.~$\pi(1)=1$. This case contributes the generating function $rz\S(t,r,1)$.
\item $A\neq\emptyset$. 
This case contributes the generating function $(\S(t,r,1)-1)tz\S(t,1,1).$
\end{itemize}
Summing up these two cases and noting that $\S(t,1,1)=N$, we deduce that
$$rz\tilde\S(t,r,1)=rz\S(t,r,1)+(\S(t,r,1)-1)tzN.$$
Solving for $\tilde\S(t,r,1)$ we get
$$\tilde\S(t,r,1)=\frac{1}{1-rz-tNz}.$$
Plugging this back into \eqref{eq:Phi-r-p}, we establish \eqref{eq:gf-312} after simplification.
\end{proof}

Recall that 
\begin{equation}\label{eq:Barn}
C:=\S(321)^{\des,\iar,\comp}(t,1,1)=\frac{1-\sqrt{1-4tz^2+4z^2-4z}}{2z(tz-z+1)},
\end{equation}
which is the generating function of the descent polynomials on $321$-avoiding permutations, first derived by Barnabei, Bonetti and Silimbani~\cite{BBS}.
\begin{theorem}\label{gf:321}
The generating function of the triple statistic $(\des,\iar,\comp)$ over $\S_n(321)$ is given by
\begin{align}\label{eq:gf-321}
\tilde \S(321)^{\des,\iar,\comp}(t,r,p)=\frac{(rpz-rz+tz)C^2-(rpz+p-1)C+p}{(1-rpz)(1-rzC)(p+C-pC)}.
\end{align}
\end{theorem}

\begin{proof}
Recently, Fu, Han and Lin~\cite[Lemma 4.5]{FHL} generalized~\eqref{eq:Barn} to 
\begin{equation*}
H:=\S(321)^{\des,\iar,\comp}(t,r,1)=\frac{1-rzC+trz^2C^2}{(1-rz)(1-rzC)}.
\end{equation*}
For convenience, let $I(r):=I_{321}(t,r)$ be the generating function over $\In_n(321)$ with respect to $(\des,\iar)$. Since $321$ is indecomposable, $\des$ is totally $\oplus$-compatible and $\iar$ is partially  $\oplus$-compatible, Eq.~\eqref{relation:comp} gives 
\begin{equation}\label{relation:comp321}
\S(321)^{\des,\iar,\comp}(t,r,p)=\frac{1}{1-rpz}+\frac{p(I(r)-rz)}{(1-pI(1))(1-rpz)}=\frac{1-p(I(1)-I(r))-rpz}{(1-pI(1))(1-rpz)}.
\end{equation}
It follows that 
$$
I(1)=1-1/C\quad\text{and}\quad I(r)-I(1)=(H/C-1)(1-rz).
$$
Substituting these back to~\eqref{relation:comp321} yields~\eqref{eq:gf-321}. 
\end{proof}


\noindent \underline{\bf Pattern $132$}
\medskip

Now we move onto the class of $132$-avoiding permutations, on which the joint distribution of $(\iar,\comp)$ is symmetric as well. We collect in the following proposition some nice features of $132$-avoiding permutations. All of the statements should be clear from the $132$-avoiding restriction, thus the proof is omitted.
\begin{proposition}\label{prop:desbot=lmin}
For any permutation $\pi\in\S_n(132)$, we have
\begin{enumerate}
\item For $2\le i\le n$, $\pi(i)$ is a descent bottom of $\pi$ if and only if it is a left-to-right minimum of $\pi$, i.e., $\LMIN(\pi)=\DESB(\pi)\cup\{\pi(1)\}$.
\item When read from left to right, the values of the left-to-right maxima of $\pi$ form a sequence of consecutive integers $\pi(1),\pi(1)+1,\pi(1)+2,\ldots,n$. 
\item The first $k = \iar(\pi)$ letters of $\pi$ equal $\pi(1),\pi(1)+1,\ldots,\pi(1)+k-1.$
\item Provided $k = \comp(\pi) \ge 2$, the last $k-1$ letters of $\pi$ equal $n-k+2,\ldots,n$.
\end{enumerate}
\end{proposition}

The next theorem strengthens Theorem~\ref{main:thm1} (iii).
\begin{theorem}\label{thm:132-symm}
For all positive integers $n$, given any two subsets $S,T\subseteq [n]$, the matrix $M_n^{\LMAX=S,\LMIN=T}(132)$ is Hankel. Consequently, the distribution of the quadruple $(\LMAX, \LMIN, \iar, \comp)$ is equal to that of $(\LMAX, \LMIN, \comp, \iar)$ over $\S_n(132)$. In terms of generating function, we have
\begin{align}\label{eq:132-LMAX-LMIN-iar-comp}
\sum_{\pi\in\S_n(132)}\mathbf{x}^{\LMAX(\pi)}\mathbf{y}^{\LMIN(\pi)}r^{\iar(\pi)}p^{\comp(\pi)} &= \sum_{\pi\in\S_n(132)}\mathbf{x}^{\LMAX(\pi)}\mathbf{y}^{\LMIN(\pi)}r^{\comp(\pi)}p^{\iar(\pi)}.
\end{align}
In particular, we have
\begin{align}\label{eq:132-des-iar-comp}
\sum_{\pi\in\S_n(132)}t^{\des(\pi)}r^{\iar(\pi)}p^{\comp(\pi)} &= \sum_{\pi\in\S_n(132)}t^{\des(\pi)}r^{\comp(\pi)}p^{\iar(\pi)}.
\end{align}
\end{theorem}
\begin{proof}
We begin by noting that if $\iar(\pi)=\comp(\pi)=1$, i.e., $\pi$ is an indecomposable $132$-avoiding permutation with $\pi(1)>\pi(2)$, then it is counted by the top-left entry of $M_n^{\LMAX=S,\LMIN=T}(132)$ for certain $S$ and $T$. Similarly, if $\iar(\pi)=\comp(\pi)=n$, then we must have $\pi=\id_n$ and it corresponds to the bottom-right entry $1$ of $M_n^{\LMAX=[n],\LMIN=\{1\}}(132)$. Otherwise, for the given subsets $S,T\subseteq [n]$, take any permutation $\pi\in\S_n(132)$ such that $\LMAX(\pi)=S$, $\LMIN(\pi)=T$, $2\le \iar(\pi)\le n-1$, and $1\le \comp(\pi)\le n-2$, we are going to pair with it a unique permutation $\sigma\in\S_n(132)$ via a bijective map $\varphi$, such that
\begin{itemize}
	\item[i.] $\pi(i)=\sigma(i)$ for $1\le i\le \iar(\pi)-1$.
	\item[ii.] $\LMAX(\sigma)=\LMAX(\pi)=S$, and $\LMIN(\sigma)=\LMIN(\pi)=T$.
	\item[iii.] $\iar(\sigma)=\iar(\pi)-1$, and $\comp(\sigma)=\comp(\pi)+1$.
\end{itemize} 
In terms of the two operations deletion and insertion that we introduce in Definition~\ref{del and ins}, we let
\begin{align*}
\sigma=\varphi(\pi):=\ins_{n,n}(\del_{\pi(1)}(\pi)),
\end{align*}
with
\begin{align*}
\pi=\varphi^{-1}(\sigma):=\ins_{\sigma(1),1}(\del_{n}(\sigma))
\end{align*}
being the inverse map. We illustrate this definition by giving an example, where the letters affected by this map have been overlined.
\medskip
$$
\begin{array}{c c c  c c c c c c c c c c c}
	&\pi &=&\bar5&\bar6&\bar7&3&4&\bar8&2&\bar9&\widebar{10}&1&\widebar{11}\\
	&\sigma &=&\bar5&\bar6&3&4&\bar7&2&\bar8&\bar9&1&\widebar{10}&\widebar{11}
\end{array}
$$

Applying Proposition~\ref{prop:desbot=lmin}, it is rountine to verify i, ii, and iii, and we leave the details to the reader. Items ii and iii ensure that $M_n^{\LMAX=S,\LMIN=T}(132)$ is Hankel as claimed. Now for any permutation $\pi\in\S_n(132)$ with $\iar(\pi)=j>\comp(\pi)=k$, we see $\tau:=\varphi^{j-k}(\pi)$ is a permutation in $\S_n(132)$ with
$$(\LMAX,\LMIN,\iar,\comp)~\tau=(\LMAX,\LMIN,\comp,\iar)~\pi.$$
Pairing permutations in this way leads to \eqref{eq:132-LMAX-LMIN-iar-comp}.

Finally, by Proposition~\ref{prop:desbot=lmin} (1) we have $\LMIN(\pi)\setminus\{\pi(1)\}=\DESB(\pi)$. Furthermore, item i above implies in particular that $\pi(1)=\sigma(1)$, combining this with $\LMIN(\pi)=\LMIN(\sigma)$ we obtain \eqref{eq:132-des-iar-comp}.
\end{proof}

\begin{theorem}\label{thm:gf-132}
We have
\begin{align}\label{eq:gf-132}
\tilde\S(132)^{\des,\iar,\comp}(t,r,p)=\dfrac{1}{1-rpz}+\dfrac{(1-z)(N-1)t}{(1-rz)(1-pz)(1-z-(N-1)tz)}.
\end{align}
\end{theorem}
\begin{proof}
The proof is analogous to that of Theorem~\ref{gf:312}. Noting that $M_n^{\des=k}(132)$ is Hankel for any fixed integer $0\le k\le n-1$ by Theorem~\ref{thm:132-symm}, we begin by interpreting this in terms of generating function. Empty permutation and identity permutations of all lengths contribute $1/(1-rpz)$, while the remaining permutations are taken care of by Lemma~\ref{gf:Hankel}, yielding
\begin{align*}
\S(t,r,p) &=\frac{1}{1-rpz}+\frac{rp}{r-p}\sum_{n\ge 2}z^n\sum_{\pi\in\In_n(132)}t^{\des(\pi)}(r^{\iar(\pi)}-p^{\iar(\pi)})\\
&=1+\frac{(rpz)^2}{1-rpz}+\frac{rp}{r-p}\sum_{n\ge 1}z^n\sum_{\pi\in\In_n(132)}t^{\des(\pi)}(r^{\iar(\pi)}-p^{\iar(\pi)}).
\end{align*}
Converting to $\tilde\S(t,r,p)$ we have
\begin{align}\label{eq:Hankel-to-Phi}
\tilde\S(t,r,p) &=\frac{rpz}{1-rpz}+\frac{r\tilde\S(t,r,0)-p\tilde\S(t,p,0)}{r-p}.
\end{align}
Plugging in $p=1$ we have
$$\tilde\S(t,r,1)=\frac{rz}{1-rz}+\frac{r\tilde\S(t,r,0)-\tilde\S(t,1,0)}{r-1}.$$
Now solve for $\tilde\S(t,r,0)$ and substitute the result back in \eqref{eq:Hankel-to-Phi} we get
\begin{align}\label{eq:tildePhi132}
\tilde\S(t,r,p)=\frac{rpz}{1-rpz}+\frac{(r-1)(\tilde\S(t,r,1)-\frac{rz}{1-rz})-(p-1)(\tilde\S(t,p,1)-\frac{pz}{1-pz})}{r-p}.
\end{align}
Next, we decompose each $132$-avoiding permutation $\pi$ as $\pi=A\:n\:B$, where $A$ and $B$ are blocks with $A>B$. In the same vein as with $312$-avoiding class, the discussion by two cases leads us to $$\tilde\S(t,r,1)=\frac{(1-z)(1+t(N-1))}{(1-rz)(1-z-tz(N-1))}.$$
We plug this back into \eqref{eq:tildePhi132} and simplify to arrive at \eqref{eq:gf-132}.
\end{proof}

\subsection{Asymmetric classes}
We deal with the three remaining classes, namely, $213$-, $231$-, and $123$-avoiding permutations. The distributions of $\iar$ and $\comp$ on each of these classes are different. We are content with deriving their joint generating functions with $\des$, and addressing a conjugate relation between $M_n(213)$ and $M_n(231)$.
\medskip

\noindent \underline{\bf Patterns $213$ and $231$}
\medskip

\begin{theorem}\label{thm:213-231-conjugate}
For every $n\ge 0$, there exists a bijection $\theta:\S_n(213)\rightarrow\S_n(231)$, such that for $\pi\in\S_n(213)$ and $\sigma:=\theta(\pi)\in\S_n(231)$, we have $\pi(1)=\sigma(1)$ and 
\begin{align*}
(\des,\iar,\comp)\:\pi = (\des,\comp,\iar)\:\sigma.
\end{align*}
In particular, the matrices $M_n(213)$ and $M_n(231)$ are conjugation of each other.
\end{theorem}
\begin{proof}
Recall the direct sum, the skew sum, and the two operations deletion and insertion that we introduce in section~\ref{sec:notation}. We define $\theta$ recursively. For $n=0,1,2$, $\theta:\S_n(213)\rightarrow\S_n(231)$ is taken to be the identity map. Now suppose $\theta$ has been defined for all $k < n~(n\ge 3)$, then take any $\pi\in\S_n(213)$, we can uniquely decompose $\pi=\pi(1)\:A\:B$ with $A>\pi(1)>B$. Now suppose $red(A)=\mu$ and $B=\nu$, then we see $\del_{\pi(1)}(\pi)=\mu\ominus\nu$, where both $\mu$ and $\nu$ are $213$-avoiding, possibly empty permutations. Let
$$\sigma:=\theta(\pi):=\ins_{\pi(1),1}(\theta(\nu)\oplus\theta(\mu)).$$
The following facts can be readily verified.
\begin{enumerate}
	\item $\sigma\in\S_n(231)$;
	\item $\sigma(1)=\pi(1)$;
	\item $\des(\sigma)=\chi(\nu\neq\emptyset)+\des(\theta(\nu))+\des(\theta(\mu))=\des(\pi)$;
	\item $\comp(\sigma)=1+\comp(\theta(\mu))=1+\iar(\mu)=\iar(\pi)$;
	\item $\iar(\sigma)=1+\chi(\nu=\emptyset)\cdot\iar(\theta(\mu))=1+\chi(\nu=\emptyset)\cdot\comp(\mu)=\comp(\pi)$.
\end{enumerate}
So we see $\sigma$ is the desired image of $\pi$, and the proof is now completed by induction.
\end{proof}

The equidistribution between $(\des,\iar,\comp)$ over $\S_n(213)$ and $(\des,\comp,\iar)$ over $\S_n(231)$ could also be drawn from comparing the following generating functions.
\begin{theorem}\label{thm:gf-213a231}
We have
\begin{align}\label{eq:gf-213}
\tilde\S(213)^{\des,\iar,\comp}(t,r,p)&=\frac{(1-rz)(tN-t+1)}{(1-rpz)(1-rz(tN-t+1))}\quad\text{and}\\
\label{eq:gf-231}
\tilde\S(231)^{\des,\iar,\comp}(t,r,p)&=\frac{(1-pz)(tN-t+1)}{(1-rpz)(1-pz(tN-t+1))}.
\end{align}
\end{theorem}
\begin{proof}
We begin with the calculation of $\tilde\S(213)^{\des,\iar,\comp}(t,r,p)$. Each $\pi\in\S_n(213)$ can be decomposed as $\pi=\pi(1)\:A\:B$, where $A>\pi(1)>B$ are $213$-avoiding blocks, possibly empty. For $n\ge 2$, we consider the following three cases:
\begin{itemize}
	\item $A=\emptyset$, $B\neq\emptyset$. This case contributes the generating function $trpz(\S(t,1,1)-1)$.
	\item $A\neq\emptyset$, $B=\emptyset$. This case contributes the generating function $rpz(\S(t,r,p)-1)$.
	\item $A\neq\emptyset$, $B\neq\emptyset$. This case contributes 
	$trpz(\S(t,r,1)-1)(\S(t,1,1)-1)$.
\end{itemize}
Summing up these three cases and noting that $\S(t,1,1)=N$, we deduce that
\begin{align*}
\S(t,r,p)=1+rpz+trpz(N-1)+rpz(\S(t,r,p)-1)+trpz(\S(t,r,1)-1)(N-1).
\end{align*}
Now we plug in $p=1$ and solve for $\S(t,r,1)$, then plug it back to deduce \eqref{eq:gf-213} after simplification. 

Decomposing each $\pi\in\S_n(231)$ as $\pi=\pi(1)\:A\:B$ with $A<\pi(1)<B$, and calculating along the same line, we can establish \eqref{eq:gf-231} as well.
\end{proof}

\noindent \underline{\bf Pattern $123$}
\medskip

For $\pi\in\S_n(123)$, clearly $\iar(\pi)\le 2$ and $\comp(\pi)\le 2$. We aim to calculate
\begin{align*}
\tilde\S(123)^{\des,\iar,\comp}(t,r,p)= A(t,p)+rB(t,p),
\end{align*}
where
\begin{align*}
pzA(t,p) &:=\sum_{n\ge 1}z^n\sum_{\pi\in\S_n(123),\:\iar(\pi)=1}t^{\des(\pi)}p^{\comp(\pi)}=pz+tpz^2+(tp+t^2p+tp^2)z^3+\cdots,\\
pzB(t,p) &:=\sum_{n\ge 2}z^n\sum_{\pi\in\S_n(123),\:\iar(\pi)=2}t^{\des(\pi)}p^{\comp(\pi)}=p^2z^2+(tp+tp^2)z^3+\cdots.
\end{align*}

By~\eqref{eq:Barn}, the generating function for the descent polynomials on $123$-avoiding permutations is
\begin{align}
\label{def:C^*}
C^* &:=\S(123)^{\des,\iar,\comp}(t,1,1)-1=\frac{\S(321)^{\des,\iar,\comp}(t^{-1},1,1;tz)-1}{t}\\
&=\frac{-1+2tz+2tz^2-2t^2z^2+\sqrt{1-4tz-4tz^2+4t^2z^2}}{2t^2z(tz-z-1)}.\nonumber
\end{align}

\begin{theorem}\label{thm:gf-123}
We have
\begin{align}
\label{eq:123A}
A(t,p) &=\frac{(p-1)tz^2}{(1-tz)^2}+\frac{(1-tz) C^*}{(1-tz+z)z},\\
\label{eq:123B}
B(t,p) &=\frac{(p-1)z}{1-tz}+\frac{ C^*}{1-tz+z}.
\end{align}
Thus, 
$$
\tilde\S(123)^{\des,\iar,\comp}(t,r,p)=\frac{(1-p)z(trz-tz-r)}{(1-tz)^2}+\frac{(1+rz-tz) C^*}{z(1+z-tz)}. 
$$
\end{theorem}
\begin{proof}
For $\pi\in\S_n(123)$ with $\iar(\pi)=1$ and $\comp(\pi)=2$, we can decompose it as $\pi=A\: B$, where $A < B$ are both decreasing subsequences with $|A|\ge 2$ and $|B|\ge 1$. On the other hand, if $\pi\in\S_n(123)$ and $\iar(\pi)=2$, then we must have $\pi(2)=n$, and we calculate the two cases $\pi(1)>\pi(3)$ and $\pi(1)<\pi(3)$ separately. All these amount to give us the functional equations:
\begin{align*}
\begin{cases}
\,A(t,p)= \frac{tpz^2}{(1-tz)^2}+\frac{ C^*}{z}-B(t,1)-\frac{tz^2}{(1-tz)^2},\\
\,B(t,p)= pz+z(A(t,1)-1)+tzB(t,p).
\end{cases}
\end{align*} 
Solving this system of equations gives rise to \eqref{eq:123A} and \eqref{eq:123B}.
\end{proof}

The following corollary can be proved combinatorially from analyzing the designated $123$-avoiding permutations. But we prove it here algebraically relying on the generating function derived in Theorem~\ref{thm:gf-123}.
\begin{corollary}
For $n\ge 2$, let $\S_n^{*}(123):=\{\pi\in\S_n(123):\des(\pi)=n-2\}$, then we have
\begin{align}\label{eq:123coro}
\sum_{n\ge 2}z^n\sum_{\pi\in\S_n^*(123)}r^{\iar(\pi)}p^{\comp(\pi)}=\frac{r^2p^2z^2}{1-z}+\frac{(r+p)rpz^3}{(1-z)^2}+\frac{(z^3+2z^4)rp}{(1-z)^2(1-2z)}.
\end{align}
In particular, the distribution of $(\iar,\comp)$ is symmetric over $\S_n^*(123)$, and the number of permutations $\pi\in\S_n^*(123)$ with $\iar(\pi)=\comp(\pi)=1$ is the sequence A095264 in \cite{oeis}.
\end{corollary}
\begin{proof}
To calculate the generating function in \eqref{eq:123coro}, we need to extract the coefficients of $t^{n-2}z^n$ in $\S(123)^{\des,\iar,\comp}(t,r,p)$ for each $n\ge 2$. For $rpzA(t,p)$, the term $\frac{(p-1)trpz^3}{(1-tz)^2}$ expands to terms all of the form $t^{n-2}z^n$, so we simply set $t=1$ to get $\frac{(p-1)rpz^3}{(1-z)^2}$, while for the term $\frac{rp(1-tz)(\S(t,1,1)-1)}{1-tz+z}$, we substitute $tz$ for $z$, and $1/t$ for $t$ in
$$\frac{rpt(1-tz) C^*}{1-tz+z},$$
then take partial derivative $\partial/\partial t$ and let $t=0$ to obtain $\frac{2rpz^3}{(1-z)^2(1-2z)}$. Similar approach yields the coefficients from $r^2pzB(t,p)$ and establishes \eqref{eq:123coro}. The claim about the symmetric distribution is evident from checking the variables $r$ and $p$ in \eqref{eq:123coro}.
\end{proof}

\section{Two patterns of length 3}\label{sec3: two 3-patterns}
In this section, we let $P=(\tau_1,\tau_2)$ be a pair of patterns of length $3$, so there are $\binom{6}{2}=15$ different pairs to consider. Once again, we accomplish two tasks as in Section~\ref{sec:one 3-pattern} and assemble our results in Table~\ref{two 3-avoiders}.

\begin{table}
{\small
\begin{tabular}{cccc}
\toprule
$P=(\tau_1,\tau_2)$ & $\tilde \S(t,r,p)$ & $M_n(P)$ & proved in\\
\midrule
&&&\\
$(132,312)$ & $\dfrac{1}{1-rpz}+\dfrac{(1-z)tz}{(1-rz)(1-pz)(1-z-tz)}$ & Hankel & Thm.~\ref{thm:gf-132-312:132-321}\\
&&&\\
$(132,321)$ & $\dfrac{1}{1-rpz}+\dfrac{tz}{(1-rz)(1-pz)(1-z)}$ & $0$-$1$ Hankel & Thm.~\ref{thm:gf-132-312:132-321}\\
&&&\\
$(213,231)$ & $\dfrac{1-z}{(1-rpz)(1-z-tz)}$ & Diagonal & Thm.~\ref{thm:gf-213-231}\\
&&&\\
$(123,312)$ & $\dfrac{1+rpz}{1-tz}+\dfrac{(r+p)tz^2}{(1-tz)^2}+\dfrac{t^2z^3}{(1-tz)^3}$ & $2\times2$ Hankel & Thm.~\ref{thm:gf-123-312}\\
&&&\\
$(213,312)$ & $\dfrac{1-rz}{(1-rpz)(1-(r+t)z)}$ & Lower triangular & Thm.~\ref{thm:gf-triple-two-3}\\
&&&\\
$(231,312)$ & $\dfrac{1-pz}{(1-rpz)(1-(p+t)z)}$ & Conjugate to $M_n(213,312)$ & Thm.~\ref{thm:gf-triple-two-3}\\
&&&\\
$
(231,321)
$ & $\dfrac{1-(1+p-t)z+(1-t)pz^2}{(1-rpz)(1-(p+1)z+(1-t)pz^2)}$ & Upper triangular & Thm.~\ref{thm:gf-triple-two-3}\\
&&&\\
$
(132,213)
$ & $\dfrac{1}{1-rpz}+\dfrac{tz}{(1-rz)(1-z-tz)}$ & Lower triangular & Thm.~\ref{thm:gf-132-conjugate}\\
&&&\\
$(132,231)$ & $\dfrac{1}{1-rpz}+\dfrac{tz}{(1-pz)(1-z-tz)}$ & Conjugate to $M_n(132,213)$ & Thm.~\ref{thm:gf-132-conjugate}\\
&&&\\
$(213,321)$ & $\dfrac{1}{1-rpz}+\dfrac{tz}{(1-z)(1-rz)(1-rpz)}$ & Lower triangular & Thm.~\ref{thm:gf-213-321}\\
&&&\\
$(312,321)$ & $\frac{1}{1-rpz}+\frac{(1-z)tz}{(1-rpz)(1-rz)(1-(1+p)z+(1-t)pz^2)}$ & No pattern & Thm.~\ref{thm:gf-312-321}\\
&&&\\
$(123,132)$ &$1+rpz+\frac{tpz^2}{1-tz}+\frac{tz(1+z-tz)(1+(r-t)z+(1-r)tz^2)}{(1-tz)((1-tz)^2-tz^2)}$ & $2\times 2$ nonzero & Thm.~\ref{thm:gf-123-four}\\
&&&\\
$(123,213)$ & $1+\dfrac{rpz}{1-tz}+\dfrac{tz(1-tz+rz)(1-tz+z)}{(1-tz)((1-tz)^2-tz^2)}$ & $2\times 2$ nonzero & Thm.~\ref{thm:gf-123-four}\\
&&&\\
$(123,231)$ & $\dfrac{1+rpz}{1-tz}+\dfrac{(1+p-tpz)tz^2}{(1-tz)^3}$ &$2\times 2$ nonzero & Thm.~\ref{thm:gf-123-four}\\
&&&\\
$(123,321)$ & $
\begin{array}{c}
1+(t+rp)z+(1+r)(1+p)tz^2\\
+(2r+t+pt)tz^3
\end{array}
$ & Ultimately zero & Thm.~\ref{thm:gf-123-four}\\
\bottomrule
\end{tabular}
}
\caption{Two patterns of length $3$}
\label{two 3-avoiders}
\vspace{-11.5pt}
\end{table}

The Wilf-classification of pairs of length $3$ patterns was done by Simion and Schmidt \cite{SS}. There are three Wilf-equivalent classes, which further split into eleven $\iar$-Wilf-equivalent subclasses: the class enumerated by $2^{n-1}$ splits into $6$ classes
\begin{align*}
&\{(132,213), (132,312),(213,231),(231,312),(231,321)\},\\
&\{(132,231)\},\,\, \{(213,312)\}, \,\,\{(312,321)\},\,\,\{(123,132)\},\,\,\{(123,213)\};
\end{align*}
the class enumerated by $1+\binom{n}{2}$ splits into $4$ classes
$$
\{(132,321)\},\,\,\{(123, 231)\},\,\,\{(213, 321)\},\,\,\{(123, 312)\};
$$
and the terminating (i.e., enumerated by $0$ when $n\ge 5$) class $\{(123,321)\}$ stays as a single class. 
For $\comp$-Wilf-equivalences, the class enumerated by $2^{n-1}$ splits into $4$ classes
\begin{align*}
&\{(132,231),(132,312),(213,231),(213,312)\},\,\,\{(132,213)\}\\
&\{(123,132),(123,213)\}, \,\,\{(231,312),(231,321),(312,321)\}
\end{align*}
and 
the class enumerated by $1+\binom{n}{2}$ splits into $2$ classes
$$
\{(132,321),(213, 321)\},\,\,\{(123, 231),(123, 312)\}.
$$
All the above refined Wilf-equivalences can be easily proven by setting $t=1$, and $p=1$ (or $r=1$) in the generating functions listed in Table~\ref{two 3-avoiders}.

\subsection{Symmetric classes} 
For $P\in\{(132,312),(132,321), (213,231), (123,312)\}$, the joint distribution of $(\des,\iar,\comp)$ is symmetric for $\iar$ and $\comp$ over $\S_n(P)$. We consider these four classes in this subsection. 

\medskip
\noindent \underline{\bf Pattern pairs $(132,312)$ and $(132,321)$}
\medskip

First note that if the pattern $312$ (resp.~$321$) occurs in a permutation $\pi$, then we can always find an occurrence of $312$ (resp.~$321$) in $\pi$ with the role of ``$3$'' played by a left-to-right maximum of $\pi$. Now recall the bijection $\varphi$ we construct in the proof of Theorem~\ref{thm:132-symm}. For each $\pi\in\S_n(132)$, observe that $\pi\in\S_n(132,312)$ (resp.~$\pi\in\S_n(132,321)$) if and only if $\varphi(\pi)\in\S_n(132,312)$ (resp.~$\varphi(\pi)\in\S_n(132,321)$). This fact, combined with Theorem~\ref{thm:132-symm}, immediately give us the following theorem.

\begin{theorem}\label{thm:132-312 and 132-321 symm}
For all positive integers $n$, given any two subsets $S,T\subseteq [n]$, the matrix $M_n^{\LMAX=S,\LMIN=T}(P)$ is Hankel, for $P\in\{(132,312), (132,321)\}$. Consequently, the distribution of the quadruple $(\LMAX, \LMIN, \iar, \comp)$ is equal to that of $(\LMAX, \LMIN, \comp, \iar)$ over $\S_n(P)$. In terms of generating function, we have
\begin{align*} 
\sum_{\pi\in\S_n(P)}\mathbf{x}^{\LMAX(\pi)}\mathbf{y}^{\LMIN(\pi)}r^{\iar(\pi)}p^{\comp(\pi)} &= \sum_{\pi\in\S_n(P)}\mathbf{x}^{\LMAX(\pi)}\mathbf{y}^{\LMIN(\pi)}r^{\comp(\pi)}p^{\iar(\pi)}.
\end{align*}
In particular, we have
\begin{align*} 
\sum_{\pi\in\S_n(P)}t^{\des(\pi)}r^{\iar(\pi)}p^{\comp(\pi)} &= \sum_{\pi\in\S_n(P)}t^{\des(\pi)}r^{\comp(\pi)}p^{\iar(\pi)}.
\end{align*}
\end{theorem}
This symmetry can also be seen directly from the following generating functions.

\begin{theorem}\label{thm:gf-132-312:132-321}
We have
\begin{align}\label{eq:gf-132-312}
\tilde\S(132,312)^{\des,\iar,\comp}(t,r,p) &=\dfrac{1}{1-rpz}+\dfrac{(1-z)tz}{(1-rz)(1-pz)(1-z-tz)},\\
\tilde\S(132,321)^{\des,\iar,\comp}(t,r,p) &=\frac{1}{1-rpz}+\frac{tz}{(1-rz)(1-pz)(1-z)}.\label{eq:gf-132-321}
\end{align}
\end{theorem}
\begin{proof}
The proof is quite analogous to that of Theorem~\ref{thm:gf-132}. First for \eqref{eq:gf-132-312}, Theorem~\ref{thm:132-312 and 132-321 symm} tells us that $M_n^{\LMAX=S,\LMIN=T}(132,312)$ is Hankel. Relying on Lemma~\ref{gf:Hankel} again, we reinterpret this in terms of generating function (details left to the readers):
\begin{align}
\label{eq:tildePhi132-312}
\tilde\S(t,r,p)=\frac{rpz}{1-rpz}+\frac{(r-1)(\tilde\S(t,r,1)-\frac{rz}{1-rz})-(p-1)(\tilde\S(t,p,1)-\frac{pz}{1-pz})}{r-p}.
\end{align}
Next, note that all $\id_n:=12\cdots n$ with $n\ge 0$ contribute collectively $1/(1-rpz)$ to $\S(132,312)^{\des,\iar,\comp}(t,r,p)$. On the other hand, every $\pi\in\S_n(132,312)$ with $\des(\pi)>0$ can be uniquely decomposed as $\pi=A\:n\:B$, where $A>B$ are two (possibly empty) blocks such that $B$ is decreasing and $A$ is $132$- and $312$-avoiding. We consider the following two cases.
\begin{itemize}
	\item $B=\emptyset$. This case contributes the generating function $pz(\S(t,r,p)-\frac{1}{1-rpz})$.
	\item $B\neq\emptyset$. This case contributes $\frac{tz}{1-tz}\cdot\frac{rpz}{1-rz}+\frac{tpz^2}{1-tz}(\S(t,r,1)-\frac{1}{1-rz})$.
\end{itemize} 
Summing up all cases gives us
$$\S(t,r,p)=\frac{1}{1-rpz}+pz(\S(t,r,p)-\frac{1}{1-rpz})+\frac{trpz^2}{(1-tz)(1-rz)}+\frac{tpz^2}{1-tz}(\S(t,r,1)-\frac{1}{1-rz}).$$
Set $p=1$ and solve to get $$\tilde\S(132,312)^{\des,\iar,\comp}(t,r,1)=\frac{1-z}{(1-rz)(1-z-tz)},$$ then plug this back into \eqref{eq:tildePhi132-312} and simplify, we get \eqref{eq:gf-132-312}. The proof of \eqref{eq:gf-132-321} is simpler noting that for $\pi\in\S_n(132,321)$ with the decomposition $\pi=A\:n\:B$, both $A$ and $B$ are increasing if $B\neq\emptyset$. The details are omitted.
\end{proof}

\medskip
\noindent \underline{\bf Pattern pair $(213,231)$}
\medskip

The first thing to notice is that for every $\pi\in\S_n(213,231)$, we must have $\pi(1)=1$ or $n$, and $\iar(\pi)=\comp(\pi)$. The latter can be proved by induction relying on the former. In terms of generating function, this means
\begin{align*}
\S(t,r,p) &= \S(t,rp,1), \; \text{and}\\
\S(t,r,p) &= 1 + rpz\S(t,r,p) + trpz(\S(t,1,1)-1).
\end{align*}
Solving these two functional equations gives us
\begin{theorem}\label{thm:gf-213-231}
$$\tilde\S(213,231)^{\des,\iar,\comp}(t,r,p)=\frac{1-z}{(1-rpz)(1-z-tz)}.$$
\end{theorem}

\medskip
\noindent \underline{\bf Pattern pair $(123,312)$}
\medskip

For every permutation $\pi\in\S_n(123,312)$, there are only five possible values for the triple $(\des(\pi),\iar(\pi),\comp(\pi))$, since $123$-avoiding implies $\iar(\pi)\le 2$ and $\comp(\pi)\le 2$, while both $312$- and $123$-avoiding forces $\des(\pi)\ge n-2$. Now it suffices to enumerate each case separately.
\begin{itemize}
	\item $(\des(\pi),\iar(\pi),\comp(\pi))=(n-1,1,1)$. There is a unique permutation $\id_n^{\mathrm{r}}=n\cdots 2 1$ for this case, which contributes $\frac{rpz}{1-tz}$ to the generating function.
	\item $(\des(\pi),\iar(\pi),\comp(\pi))=(n-2,2,2)$. There is a unique permutation $1\oplus\id_{n-1}^{\mathrm{r}}$ for this case, which contributes $\frac{r^2p^2z^2}{1-tz}$ to the generating function.
	\item $(\des(\pi),\iar(\pi),\comp(\pi))=(n-2,1,1)$. Permutations in this case are of the form $\pi=a\:a-1\cdots b\:n\cdots a+1\:b-1\cdots 1$, where $1<b<a<n$. Therefore this case contributes $\frac{t^2rpz^4}{(1-tz)^3}$ to the generating function.
	\item $(\des(\pi),\iar(\pi),\comp(\pi))=(n-2,2,1)$ or $(n-2,1,2)$. These two cases can be discussed similarly as the last case, and the contributions are $\frac{tr^2pz^3}{(1-tz)^2}$ and $\frac{trp^2z^3}{(1-tz)^2}$.
\end{itemize}
Summing up all cases above gives rise to
\begin{theorem}\label{thm:gf-123-312}
\begin{equation*}
\tilde\S(123,312)^{\des,\iar,\comp}(t,r,p)=\frac{1+rpz}{1-tz}+\frac{(r+p)tz^2}{(1-tz)^2}+\frac{t^2z^3}{(1-tz)^3}.
\end{equation*}
\end{theorem}

\subsection{Asymmetric classes}
For the remaining choices of $P$, the distribution of $(\iar,\comp)$ over $\S_n(P)$ is not symmetric. However, we still observe some conjugative pairs as in Section~\ref{sec:one 3-pattern}.

\medskip
\noindent \underline{\bf Pattern pairs $(213,312)$, $(231,312)$ and $(231,321)$}
\medskip

Recall the two bijections, $\xi$ from Theorem~\ref{thm:321=312}, and $\theta$ from Theorem~\ref{thm:213-231-conjugate}. Observe that
\begin{itemize}
	\item $\pi\in\S_n(231,321)$ if and only if $\xi(\pi)\in\S_n(231,312)$.
	\item $\pi\in\S_n(213,312)$ if and only if $\theta(\pi)\in\S_n(231,312)$.
\end{itemize}
Then the following theorem is a quick corollary of Theorems~\ref{thm:321=312} and \ref{thm:213-231-conjugate}.

\begin{theorem}\label{thm:triple-two-3}
For each $n\ge 1$, the quadruple $(\LMAX,\LMAXP,\iar,\comp)$ has the same distribution over $\S_n(231,321)$ and $\S_n(231,312)$; the distribution of the triple $(\des,\iar,\comp)$ over $\S_n(213,312)$ is equal to that of $(\des,\comp,\iar)$ over $\S_n(231,312)$.
\end{theorem}

Next, we compute the generating functions for these three pairs.
\begin{theorem}\label{thm:gf-triple-two-3}
We have
\begin{align}
\label{eq:213-312}
\tilde\S(213,312)^{\des,\iar,\comp}(t,r,p) &=\frac{1-rz}{(1-rpz)(1-(r+t)z)},\\
\label{eq:231-312}
\tilde\S(231,312)^{\des,\iar,\comp}(t,r,p) &=\frac{1-pz}{(1-rpz)(1-(p+t)z)},\\
\label{eq:231-321}
\tilde\S(231,321)^{\des,\iar,\comp}(t,r,p) &=\frac{1-(1+p-t)z+(1-t)pz^2}{(1-rpz)(1-(p+1)z+(1-t)pz^2)}.
\end{align}
\end{theorem}
\begin{proof}
In view of Theorem~\ref{thm:triple-two-3}, \eqref{eq:213-312} follows from \eqref{eq:231-312} by switching variables $r$ and $p$. To prove \eqref{eq:231-312}, note that both patterns $231$ and $312$ are indecomposable, thus we can apply Lemma~\ref{general} to reduce the calculation to that of the generating function of $(\des,\iar)$ over $\In_n(231,312)$. But the indecomposable permutations in $\S_n(231,312)$ are precisely $\id_n^{\mathrm{r}}=n\:n-1\cdots 1$, thus $I_{231,312}(t,r)=\frac{rz}{1-tz}$. Plugging this back into \eqref{relation:comp} gives us \eqref{eq:231-312}. Finally, every permutation in $\In_n(231,321)$ must be of the form $1\ominus\id_{n-1}=n\:1\:2\cdots n-1$, yeilding the generating function $I_{231,321}(r,t)=rz+\frac{trz^2}{1-z}$. Applying \eqref{relation:comp} from Lemma~\ref{general} again, we derive \eqref{eq:231-321} and complete the proof.
\end{proof}

\medskip
\noindent \underline{\bf Pattern pairs $(132,213)$ and $(132,231)$}
\medskip

For the same reason that the bijection $\theta$ from Theorem~\ref{thm:213-231-conjugate} preserves the $132$-avoidance, we have the following conjugate relation.

\begin{theorem}\label{thm:132-conjugate}
For each $n\ge 1$, the distribution of the triple $(\des,\iar,\comp)$ over $\S_n(132,213)$ is equal to that of $(\des,\comp,\iar)$ over $\S_n(132,231)$.
\end{theorem}

Next, note that each permutation $\pi\in\S_n(132,231)$ either begins with $\pi(1)=n$, or ends with $\pi(n)=n$. Calculating these two cases separately we have
$$\S(t,r,p)=\frac{1}{1-rpz}+pz(\S(t,r,p)-\frac{1}{1-rpz})+trpz(\S(t,1,1)-1).$$
Solving this and applying Theorem~\ref{thm:132-conjugate}, we can deduce the following theorem.

\begin{theorem}\label{thm:gf-132-conjugate}
We have
\begin{align*}
\tilde\S(132,213)^{\des,\iar,\comp}(t,r,p) &=\dfrac{1}{1-rpz}+\dfrac{tz}{(1-rz)(1-z-tz)},\\
\tilde\S(132,231)^{\des,\iar,\comp}(t,r,p) &=\dfrac{1}{1-rpz}+\dfrac{tz}{(1-pz)(1-z-tz)}.
\end{align*}
\end{theorem}

\medskip
\noindent \underline{\bf Pattern pair $(213,321)$}
\medskip

Note that each permutation $\pi\in\S_n(213,321)$ can be decomposed as $\pi=A\:n\:B$, where $A$ and $B$ are both increasing blocks, and $B$ is consisted of consecutive integers. Calculating the two cases $1\in A$ and $1\in B$ separately, we obtain the following theorem.
\begin{theorem}\label{thm:gf-213-321}
We have
\begin{equation*}
\tilde \S(213,321)^{\des,\iar,\comp}(t,r,p)=\dfrac{1}{1-rpz}+\dfrac{tz}{(1-z)(1-rz)(1-rpz)}.
\end{equation*}
\end{theorem}

\medskip
\noindent \underline{\bf Pattern pair $(312,321)$}
\medskip

Noting that both $312$ and $321$ are indecomposable patterns, we apply \eqref{relation:comp}
\begin{equation*}
F_P(q)=\frac{1}{1-qw}+\frac{q(I_P({\bf t}, {\bf t'})-w)}{(1-qI_P({\bf t},{\bf 1}))(1-qw)}
\end{equation*}
from Lemma~\ref{general} (1) to reduce the calculation to
\begin{equation*}
I_{312,321}(t,r):=\sum_{n\ge 1}z^n\sum_{\substack{\pi\in\S_n(312,321) \\ \comp(\pi)=1}}t^{\des(\pi)}r^{\iar(\pi)}.
\end{equation*}
Now any indecomposable $\pi\in\S_n(312,321)$ must be of the form $\pi=2\:3\cdots n\:1$. Hence $$I_{312,321}(t,r)=rz+\frac{trz^2}{1-rz},$$
with which we can deduce
\begin{theorem}\label{thm:gf-312-321}
\begin{equation*}
\tilde\S(312,321)^{\des,\iar,\comp}(t,r,p)=\frac{1}{1-rpz}+\frac{(1-z)tz}{(1-rpz)(1-rz)(1-(1+p)z+(1-t)pz^2)}.
\end{equation*}
\end{theorem}

\medskip
\noindent \underline{\bf Pattern pairs $(123,132)$, $(123,213)$, $(123,231)$ and $(123,321)$}
\medskip

These four pattern sets all contain the pattern $123$, hence $\iar(\pi)\le 2$ and $\comp(\pi)\le 2$ for each permutation $\pi$ in $\S_n(P)$. We take similar approach as Theorem~\ref{thm:gf-123}, or analyze the position of $1$ or $n$ in $\pi$, to calculate their generating functions. We collect the results in the following theorem but omit the proof.

\begin{theorem}\label{thm:gf-123-four}
We have
\begin{align*}
\tilde\S(123,132)^{\des,\iar,\comp}(t,r,p) &=1+rpz+\frac{tpz^2}{1-tz}+\frac{tz(1+z-tz)(1+(r-t)z+(1-r)tz^2)}{(1-tz)((1-tz)^2-tz^2)},\\
\tilde\S(123,213)^{\des,\iar,\comp}(t,r,p) &=1+\frac{rpz}{1-tz}+\frac{tz(1-tz+rz)(1-tz+z)}{(1-tz)((1-tz)^2-tz^2)},\\
\tilde\S(123,231)^{\des,\iar,\comp}(t,r,p) &=\dfrac{1+rpz}{1-tz}+\dfrac{(1+p-tpz)tz^2}{(1-tz)^3},\\
\tilde\S(123,321)^{\des,\iar,\comp}(t,r,p) &=1+(t+rp)z+(1+r)(1+p)tz^2+(2r+t+pt)tz^3.
\end{align*}
\end{theorem}


\section{Schr\"oder classes: two patterns of length 4}
\label{sec:schroder}

This section aims to characterize the pattern pair $P$ of length $4$ whose distribution matrix $M_n(P)$ equals  $M_n(2413,3142)$. 
The first few values of the symmetric matrices  $M_n(2413,3142)$ are: 
$$
\begin{bmatrix}\label{tri:Sch2}
1 & 0 \\
0 & 1 
\end{bmatrix},
\begin{bmatrix}\label{tri:Sch3}
2 & 1 & 0 \\
1 & 1  & 0\\
0 & 0 & 1
\end{bmatrix},
\begin{bmatrix}\label{tri:Sch4}
7 & 3 & 1&0 \\
3 & 3  & 1&0\\
1 & 1 & 1&0\\
0&0&0&1
\end{bmatrix},
\begin{bmatrix}\label{tri:Sch5}
28 & 12 & 4&1 &0\\
12 & 11  & 4&1&0\\
4 & 4& 3&1&0\\
1&1&1&1&0\\
0&0&0&0&1
\end{bmatrix},
\begin{bmatrix}\label{tri:Sch}
121 & 52 & 18&5 &1&0\\
52 & 46 & 17&5&1&0\\
18 & 17& 12&4&1&0\\
5&5&4&3&1&0\\
1&1&1&1&1&0\\
0&0&0&0&0&1
\end{bmatrix}.
$$
The integer sequence formed by the entries in the upper-left corner of $M_n(2413,3142)$ begins with 
$$
1,1,2,7,28,121,550, 2591,\ldots.
$$
This sequence appears to match A010683 in the OEIS~\cite{oeis}, a sequence that counts, among many combinatorial objects, dissections of a convex polygon with $n+3$ sides having a triangle over a fixed side (the base) of the polygon. 
This coincidence can be proved by comparing $\tilde \S(\SE)(1,0,0)$ from the expression~\eqref{eq:dou-sch} with the generating function supplied in the entry A010683.

The first result in this section is a consequence of Theorems~\ref{thm:sep:sym} and~\ref{thm:sep}. 

\begin{corollary} 
\label{comp:iar}
For $n\geq1$,
\begin{equation}\label{eq:iar:comp}
\sum_{\pi\in\S_n(2413,3142)}t^{\des(\pi)}x^{\comp(\pi)}y^{\iar(\pi)}=\sum_{\pi\in\S_n(2413,4213)}t^{\des(\pi)}x^{\comp(\pi)}y^{\iar(\pi)}.
\end{equation} 
Consequently, 
\begin{equation}\label{sym4213}
\sum_{\pi\in\S_n(2413,4213)}t^{\des(\pi)}x^{\comp(\pi)}y^{\iar(\pi)}=\sum_{\pi\in\S_n(2413,4213)}t^{\des(\pi)}x^{\iar(\pi)}y^{\comp(\pi)}
\end{equation}
\end{corollary}
\begin{proof}
Since the patterns $2413,4213$ and $3142$ are indecomposable, the equidistribution~\eqref{eq:iar:comp} is a consequence of Theorem~\ref{thm:sep} (with $x=1$) and Lemma~\ref{general}.

The identity~\eqref{sym4213} follows directly from~\eqref{eq:iar:comp} and Theorem~\ref{thm:sep:sym}.
\end{proof}

\begin{remark}
In view of Corollary~\ref{comp:iar}, one may wonder that if~\eqref{eq:sep} can be further refined by $\comp$. This is not true and  it turns out that even $(\dd,\comp)$ is not equidistributed over $\S_5(2413,3142)$ and $\S_5(2413,4213)$. Can Theorem~\ref{thm:sep} be further refined by other classical  permutation statistics (cf.~\cite{kit})?
\end{remark}

Lin and Kim~\cite{LK} showed that, among all permutation classes avoiding two patterns of length 4, the three classes below  are the only nontrivial classes which are $\des$-Wilf equivalent to the class of separable permutations.

\begin{theorem}[\text{Lin and Kim~\cite[Theorem~5.1]{LK}}]\label{thm:kim}
We have the refined Wilf-equivalences:
$$
(2413,3142)\sim_{\des}(2413,4213)\sim_{\DES}(2314,3214)\sim_{\DES}(3412,4312).
$$
\end{theorem}

It should be noted that $\iar$ is not a Comtet statistic over $\S_n(2314,3214)$. Computer program indicates that, among all permutation classes avoiding two patterns of length 4,  the classes of $(2413,4213)$ and $(3412,4312)$ are the only two that are $(\des,\iar,\comp)$-Wilf equivalent to the class of separable permutations.
\begin{theorem}\label{thm:3412-4312}
We have the refined Wilf-equivalence $(2413,4213)\sim_{(\DES,\comp)}(3412,4312)$. In particular, 
$$(2413,3142)\sim_{(\des,\iar,\comp)}(2413,4213)\sim_{(\des,\iar,\comp)}(3412,4312).$$ Consequently, 
\begin{equation}\label{sym3412}
\sum_{\pi\in\S_n(3412,4312)}t^{\des(\pi)}x^{\comp(\pi)}y^{\iar(\pi)}=\sum_{\pi\in\S_n(3412,4312)}t^{\des(\pi)}x^{\iar(\pi)}y^{\comp(\pi)}.
\end{equation}
\end{theorem}


In order to prove Theorem~\ref{thm:3412-4312}, we need a set-valued version of Lemma~\ref{general}. For an integer $\ell$ and a set $S=\{s_1,s_2,\ldots\}$, let $\ell\oplus S:=\{\ell+s_1,\ell+s_2,\ldots\}$. A set-valued statistic $\ST$ is called {\em totally $\oplus$-compatible} if for each $\pi=\tau_1\oplus\tau_2\oplus\cdots\oplus\tau_k$ with each $\tau_i$ an indecomposable permutation of length $\ell_i$, 
$$
\ST(\pi)=\bigcup_{i=1}^k c_i\oplus\ST(\tau_i),
$$
where $c_i=\sum_{j=1}^{i-1}\ell_j$. Note that the set-valued statistics $\DES$, $\DESB$, $\LMAX$ and $\LMAXP$ are all totally $\oplus$-compatible. 

\begin{lemma}\label{gen:setv}
Let $(\ST_1,\ST_2,\ldots)$ be a sequence of  totally  $\oplus$-compatible set-valued statistics.
Let $P$ and $Q$ be two collections of indecomposable patterns. If $(\ST_1,\ST_2,\ldots)$ has the same distribution over $\S_n(P)$ and $\S_n(Q)$ for $n\geq1$, then so does $(\comp,\ST_1,\ST_2,\ldots)$.
\end{lemma}
\begin{proof}
Since $P/Q$ is a collection of indecomposable patterns, each $P/Q$-avoiding  permutation is a direct sum of some smaller $P/Q$-avoiding  permutations. Thus, it is sufficient to show that if $(\ST_1,\ST_2,\ldots)$ is equidistributed over $\S_n(P)$ and $\S_n(Q)$ for $n\geq1$, then $(\ST_1,\ST_2,\ldots)$ is equidistributed over $\In_n(P)$ and $\In_n(Q)$. We aim to prove this by induction on $n$. 

Obviously, the assertion is true for $n=1$. Suppose that $(\ST_1,\ST_2,\ldots)$ is equidistributed over $\In_n(P)$ and $\In_n(Q)$ for $n\leq m$. It follows that  $(\ST_1,\ST_2,\ldots)$ is equidistributed over $\S_{m+1}\setminus\In_{m+1}(P)$ and $\S_{m+1}\setminus\In_{m+1}(Q)$, as $(\ST_1,\ST_2,\ldots)$ is a sequence of  totally  $\oplus$-compatible set-valued statistics. Now $(\ST_1,\ST_2,\ldots)$ is also equidistributed over $\S_{m+1}(P)$ and $\S_{m+1}(Q)$ and so $(\ST_1,\ST_2,\ldots)$ is equidistributed over $\In_{m+1}(P)$ and $\In_{m+1}(Q)$. This completes the proof by induction. 
\end{proof}

\begin{proof}[{\bf Proof of Theorem~\ref{thm:3412-4312}}]
The refined Wilf-equivalence $(2413,4213)\sim_{(\DES,\comp)}(3412,4312)$ is a direct consequence of Theorem~\ref{thm:kim} and Lemma~\ref{gen:setv}, as the set-valued statistic $\DES$ is totally $\oplus$-compatible. The other two statements then follow immediately from Corollary~\ref{comp:iar}. 
\end{proof}

Next we compute the generating function $\tilde\S(\SE)(t,r,p)=(\S(\SE)(t,r,p;z)-1)/rpz$ with respect to $(\des,\iar,\comp)$, where $\SE$ is a pattern pair in $\{(2413,3142),(2413,4213),(3412,4312)\}$.

\begin{theorem}Let $S(t):=\S(\SE)(t,1,1;z)-1$. Then,
\begin{align}\label{eq:dou-sch}
\tilde\S(\SE)(t,r,p) &= \frac{(1/z+1-r-p)S(t)+(1-r)(1-p)S(t)^2}{(1-rpz)(1+(1-p)S(t))(1+(1-r)S(t))},
\end{align}
where $S(t)$ satisfies the algebraic functional equation
\begin{equation}\label{eq:sch-des}
S(t)=z+(1+t)zS(t)+tzS(t)^2+tS(t)^3.
\end{equation}
\end{theorem}
\begin{proof}
The functional equation~\eqref{eq:sch-des} for the generating function $S(t)$ of the descent polynomials over separable permutations was proved in~\cite{flz}. Since the patterns $2413$ and $3142$ are indecomposable, $\des$ is totally $\oplus$-compatible and $\iar$ is partially $\oplus$-compatible, Eq.~\eqref{gen:symm} gives 
\begin{equation}\label{gen:symmS}
\S(\SE)(t,r,p;z)=\frac{1-rpz+(rpz+rp-r-p)I_{\SE}}{(1-rI_{\SE})(1-pI_{\SE})(1-rpz)},
\end{equation}
where $I_{\SE}:=I_{\SE}(t;z)$ is the generating function with respect to $\des$. Since $S(t)=\frac{I_{\SE}}{1-I_{\SE}}$, we have 
$$
I_{\SE}=\frac{S(t)}{1+S(t)}.
$$
Substitute this into~\eqref{gen:symmS} and simplify, we get~\eqref{eq:dou-sch}.  
\end{proof}

Aided by the computer program, we make the following conjecture, whose validity will complete the characterization of pattern pairs of length $4$ that are  $\iar$-Wilf equivalent to the class of separable permutations.
\begin{conj}\label{schroder:iar}
Let $P\notin\{(2413,3142),(2413,4213),(3412,4312)\}$ be a pair of patterns of length $4$. Then,
$P$ is $\iar$-Wilf equivalent to $(2413,4213)$
  if and only if $P$ is one of the following eleven  pairs:
 \begin{align*}
 &(1324,2134),(1324,3124),(1423,4123),(1432,4132),(2134,2314),(2314,3124) \nonumber\\
 &\,\,\qquad\blue{(2431,4231)},(2431,3241),(3241,3421),(3421,4231),(3421,4321).
 \end{align*}
 Moreover, if $P$ is one of the last five pairs (i.e.,~those in the second line above), then $P$ is $(\iar,\comp)$-Wilf equivalent to $(2413,4213)$. 
\end{conj}
\begin{remark}
In view of Lemma~\ref{general}, the second assertion for the $(\iar,\comp)$-Wilf equivalences in Conjecture~\ref{schroder:iar} follows automatically from the first assertion, as all the patterns appear in the last five pairs are indecomposable.   
\end{remark}
In the rest of this section, we aim to confirm Conjecture~\ref{schroder:iar} for  the pattern pair $P=(2431,4231)$ using the technique of generating trees, which was originally employed  to study the {\em Baxter permutations} by  Chung, Graham, Hoggatt and Kleiman~\cite{chung}, see also \cite{SS,wes}.

\begin{theorem}\label{sch:last}
We have the refined Wilf equivalence 
$$(2413,4213)\sim_{(\LMAXP,\comp)}(2431,4231).$$
In particular, $(2413,4213)\sim_{(\lmax,\iar,\comp)}(2431,4231)$ and Conjecture~\ref{schroder:iar} is true for  $P=(2431,4231)$. 
\end{theorem}

In view of Lemma~\ref{gen:setv}, to prove Theorem~\ref{sch:last}, it is sufficient to prove the refined Wilf-equivalence $(2413,4213)\sim_{\LMAXP}(2431,4231)$. We will prove this by showing a growth  rule for $(2431,4231)$-avoiding permutations and then comparing  it with that of $(2413,4213)$-avoiding permutations. 

For $\pi\in\S_{n-1}$ and $i\in[n]$, let $\ins_i(\pi):=\ins_{i,n}(\pi)\in\S_{n}$. For example, $\ins_3(14532)=156423$. If $\pi\in\S_{n-1}(2431,4231)$, then introduce the set of {\em{\bf\em ava}ilable inserting values} of $\pi$ as
$$
\AVA(\pi):=\{k\in[n]: \ins_k(\pi)\in\S_{n}(2431,4231)\}=\{k_1>k_2>\cdots\}.
$$
 Clearly, if $i\in\AVA(\pi)$, then $k\in\AVA(\pi)$ for any $i\leq k\leq n$, since the newly inserted letter, which appears at the end, can only play the role of `1' in a pattern $2431$ or $4231$. Thus, $\AVA(\pi)=[m,n]:=\{m,m+1,\ldots,n\}$ for some $m<n$. We will call $m$ the {\em critical value} of $\pi$ in the sequel.
For example, we have $\AVA(14523)=[3,6]$.

We have the following growth rule for $(2431,4231)$-avoiding permutations. 
\begin{lemma}\label{insert:2431}Suppose $\pi\in\S_{n-1}(2431,4231)$ with 
$\AVA(\pi)=[m,n]$.
Then, 
$$
\AVA(\ins_j(\pi))=
\begin{cases}
[j,n+1], & \text{if $m\leq j\leq n-1$;}\\
[m,n+1], & \text{if $j=n$.} 
\end{cases}
$$
\end{lemma}
\begin{proof}
For  $m\leq j\leq n-1$, the letters $j-1$ (if $j\ge 2$) and $j+1$ appear before $j$ in $\ins_j(\pi)$ and these three letters form a pattern $132$ or $312$. Thus, $j-1\notin\AVA(\ins_j(\pi))$. On the other hand, suppose $\hat\pi:=\ins_j(\ins_j(\pi))=\hat\pi(1)\cdots\hat\pi(n)\hat\pi(n+1)$, then we see $\hat\pi(a),\hat\pi(b),\hat\pi(c)$ and $\hat\pi(n)=j+1$ form a pattern $2431$ or $4231$, if and only if $\hat\pi(a),\hat\pi(b),\hat\pi(c)$ and $\hat\pi(n+1)=j$ do. This means we have $j\in\AVA(\ins_j(\pi))$. Therefore $j$ is the critical value of $\ins_j(\pi)$ and $\AVA(\ins_j(\pi))=[j,n+1]$. Clearly, $\AVA(\ins_n(\pi))=[m,n+1]$. This completes the proof of the lemma. 
\end{proof}


The definition of $\AVA(\pi)$ for a $(2413,4213)$-avoiding permutation $\pi$ was introduced similarly in~\cite{LK}, where they proved the following growth rule. Note that for any $\pi\in\S_{n-1}(2413,4213)$, $\AVA(\pi)$ always contains $1$ and $n$.
\begin{lemma}[\text{Lin and Kim~\cite[Lemma~5.3]{LK}}]
\label{insert:2413}Suppose $\pi\in\S_{n-1}(2413,4213)$ with 
$$
\AVA(\pi)=\{n=k_1>k_2>\cdots>k_m=1\}.
$$
Then, for  $1\leq j\leq m$,
$$
\AVA(\ins_{k_j}(\pi))=\{n+1\geq k_j+1>k_j>k_{j+1}>\cdots>k_m=1\}.
$$
\end{lemma}
 
 We are ready to prove Theorem~\ref{sch:last} by constructing the generating trees for both classes. 
\begin{proof}[{\bf Proof of Theorem~\ref{sch:last}}] Label each $\pi\in\S_n(2431,4231)$ by $|\AVA(\pi)|$, then Lemma~\ref{insert:2431} produces the rewriting rule:
\begin{align}\label{rule}
\Omega_{\mathrm{Sch}}=\begin{cases}
(2) \\
(k) \leadsto (k+1), (k+1), (k), (k-1), \ldots, (3).
\end{cases}
\end{align}
This means that the initial permutation $\id_1$ has label $(2)$ and all the $(2431,4231)$-avoiding permutations derived from inserting a letter at the end of a $(2431,4231)$-avoiding permutation labeled by $(k)$, are exactly those with labels $(k+1), (k+1), (k), (k-1), \ldots, (3)$.

\begin{figure}
\setlength{\unitlength}{1mm}
\begin{picture}(120,30)\setlength{\unitlength}{1mm}
\thinlines
\put(50,25){$(2)^*$}
\put(52,23.5){\line(-3,-1){15}}\put(53,23.5){\line(3,-1){15}}
\put(34,14){$(3)^*$}    \put(66,14){$(3)$}
\put(36,12.5){\line(-1,-1){8}}\put(36.5,12.5){\line(0,-1){8}}\put(37,12.5){\line(1,-1){8}}
\put(25,1){$(4)^*$}\put(34,1){$(4)$}\put(43,1){$(3)$}

\put(68,12.5){\line(-1,-1){8}}\put(68.5,12.5){\line(0,-1){8}}\put(69,12.5){\line(1,-1){8}}
\put(57,1){$(4)^*$}\put(66,1){$(4)$}\put(75,1){$(3)$}
\end{picture}
\caption{First three levels of the generating tree for $\cup_{n\geq1}\S_n(2431,4231)$.}
\label{tree}
\end{figure}
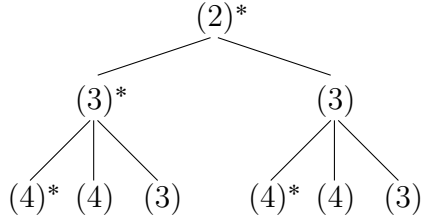

We can construct a {\em generating tree} (an infinite rooted and labeled tree) for $(2431,4231)$-avoiding permutations by representing each permutation as a node on the tree using its label. More precisely, the root is labeled $(2)$, and the children of a node labeled $(k)$ are those generated according to the rewriting rule $\Omega_{\mathrm{Sch}}$ in~\eqref{rule}. In addition, the labels for those permutations ending with their greatest letter will have an extra `$*$', and we will call the corresponding nodes the {\em star nodes}. So in this generating tree, every node has precisely one child being a star node. See Fig.~\ref{tree} for the first few levels of this generating tree. Note that the nodes at the $n$-th level of this tree are in one-to-one correspondence with elements of $\S_n(2431,4231)$. Moreover, if a permutation $\pi\in\S_n(2431,4231)$ is labeled by $\ell(\pi)$, and the unique path from the root $(2)^*$ to $\ell(\pi)$ goes through $p_1=(2)^*, p_2,\ldots,p_n=\ell(\pi)$, then 
$$
\LMAXP(\pi)=\{i: p_i \text{ is a star node}\}.
$$
For instance, the second $(4)^*$ appearing in level $3$ corresponds to $213$ and $\LMAXP(213)=\{1,3\}$. In other words, the distribution of $\LMAXP$ over $(2431,4231)$-avoiding permutations is completely  determined by this generating tree. 

It can be readily checked that Lemma~\ref{insert:2413} gives the same rewriting rule $\Omega_{\mathrm{Sch}}$ for $(2413,4213)$-avoiding permutations, which in turn, produces for $(2413,4213)$-avoiding permutations the identical generating tree as $(2431,4231)$-avoiding permutations. This proves $(2413,4213)\sim_{\LMAXP}(2431,4231)$, as desired. 
\end{proof}

\section{Revisiting separable and (2413,4213)-avoiding permutations}
\label{sec:revisit}

The main purpose of this section is to prove Theorem~\ref{thm:sep}. We begin with the motivation that leads to the discovery of Theorem~\ref{thm:sep}.

Recall that a sequence $e=(e_1,e_2,\ldots,e_n)\in\N^n$ is an {\em inversion sequence} of length $n$ if $e_i<i$ for each $i\in[n]$. An inversion sequence is {\em$021$-avoiding} if its positive entries are weakly increasing. Denote by $\I_n(021)$ the set of $021$-avoiding inversion sequences of length $n$.  Kim and Lin~\cite{LK} 
\begin{itemize}
\item constructed a bijection $\Psi:\I_n(021)\rightarrow\S_n(2413,4213)$ which transforms the set-valued statistic $\ASC$ to $\DES$, where $\ASC(e):=\{i\in[n-1]: e_i<e_{i+1}\}$ is the set of ascents of $e$. In particular,  together with the works in~\cite{cor,flz,lin} we know  
\begin{equation}\label{int:sch}
S_n(t)=\sum_{e\in\I_n(021)}t^{\asc(e)}=\sum_{\pi\in\S_n(2413,4213)}t^{\des(\pi)},
\end{equation}
where $\asc(e):=|\ASC(e)|$;
\item proved combinatorially via the so-called {\em modified Foata--Strehl action} that
\begin{equation}\label{gam:act}
\sum_{\pi\in\S_n(2413,4213)}t^{\des(\pi)}=\sum_{k=0}^{\lfloor\frac{n-1}{2}\rfloor}|\Gamma_{n,k}(2413,4213)|t^k(1 + t)^{n-1-2k}.
\end{equation}
\end{itemize}
Recall that $\Gamma_{n,k}(2413,4213)$ is the set of permutations in $\S_n(2413,4213)$ with $k$ descents and without double descents. Combining~\eqref{gam:sch},~\eqref{int:sch} and~\eqref{gam:act} yields 
\begin{equation}\label{eq:gam}
|\Gamma_{n,k}(2413,3142)|=|\Gamma_{n,k}(2413,4213)|
\end{equation} 
for all $0\leq k\leq n-1$. This identity was refined recently by Wang~\cite{wan} as
\begin{equation}\label{eq:wang}
\sum_{\pi\in\S_n(2413,3142)}t^{\des(\pi)}x^{\dd(\pi)}=\sum_{\pi\in\S_n(2413,4213)}t^{\des(\pi)}x^{\dd(\pi)},
\end{equation} 
where $\dd(\pi)$ denotes the number of double descents of $\pi$. Setting $x=0$ in~\eqref{eq:wang} we recover~\eqref{eq:gam}. 

Theorem~\ref{thm:sep} is a refinement of 
Wang's equidistribution~\eqref{eq:wang} by the Comtet statistic $\iar$. 
The three numerical statistics $\des$, $\dd$ and $\iar$ are all determined by the set-valued statistic $\DES$, but unfortunately $(2413,4213)$ is not $\DES$-Wilf equivalent to $(2413,3142)$. In spite of that, we still have the refined Wilf-equivalence $(2413,4213)\sim_{(\des,\dd,\iar)}(2413,3142)$, to our surprise. 
Our proof of Theorem~\ref{thm:sep} is purely algebraic, basing on Kim--Lin's bijection $\Psi$, a decomposition of $021$-avoiding inversion sequences and Stankova's block decomposition of separable permutations~\cite{sk}. It would be interesting to construct a bijective proof of this equidistribution. 

As we will see, some easy combinatorial arguments on $021$-avoiding inversion sequences (see Theorem~\ref{rec:021}) together with Theorem~\ref{thm:sep} provide an alternative approach to a recent result of the first and third authors~\cite[Theorem~3.2]{FW}.

\subsection{A recurrence for 021-avoiding inversion sequences}
\label{sec:021}
For each inversion sequence $e=(e_1,\ldots,e_n)$, let $\izero(e):=\min(\ASC(e)\cup\{n\})$ be the number of {\em initial zeros} of $e$. It follows from the aforementioned bijection $\Psi$ that for $1\leq k\leq n$,
\begin{equation}
I_{n,k}:=|\{e\in\I_n(021):\izero(e)=k\}|=|\{\pi\in\S_n(2413,4213):\iar(\pi)=k\}|.
\end{equation}
Thus, 
$$
I_{n,k}=|\{\pi\in\S_n(2413,3142):\iar(\pi)=k\}|
$$
by Theorem~\ref{thm:sep}. 
We have the following recurrence relation for $I_{n,k}$. 
\begin{theorem} \label{rec:021}We have $I_{1,1}=1$ and 
\begin{align}
\label{k1}I_{n,1}&=\sum_{k=1}^{n-1}2^{k-1}I_{n-1,k}\quad\text{for $n\geq2$},\\
\label{k2}I_{n,i}&=I_{n-1,i-1}+\sum_{k=i}^{n-1}2^{k-i}I_{n-1,k}\quad\text{for $2\leq i \leq n$}.
\end{align}
\end{theorem}
\begin{proof}
Let $\I_{n,k}:=\{e\in\I_n(021):\izero(e)=k\}$.
For each inversion sequence $e\in\I_n(021)$, let $\delta(e)=(\bar e_2,\bar e_3,\ldots,\bar e_n)\in\I_{n-1}(021)$ with $\bar e_i=e_i-\chi(e_i>0)$ for $2\leq i\leq n$. The mapping $\delta:\I_n(021)\rightarrow\I_{n-1}(021)$ is surjective. To see~\eqref{k1}, for any $e\in\I_{n-1}(021)$ with $\izero(e)=k$, there are exactly $2^{k-1}$ preimages of $e$ in $\I_{n,1}$ under $\delta$, because 
\begin{itemize}
\item each of the $k$ initial zeros of $e$, except for the first zero, can be either $0$ or $1$ in its preimages;
\item but all zeros after the first positive entry of $e$, must remain zeros in its preimages, to guarantee that they are $021$-avoiding. 
\end{itemize}
Recursion~\eqref{k2} follows from similar reasoning. 
\end{proof}

\subsection{Proof of Theorem~\ref{thm:sep}} We will prove Theorem~\ref{thm:sep} by showing that the generating functions for both sides of~\eqref{eq:sep} satisfy the same algebraic functional equation. We begin with the calculation of the generating function for the right-hand side of~\eqref{eq:sep}:
\begin{align*}
G(t,x,y;z)&:=\sum_{n\geq1}z^n\sum_{\pi\in\S_n(2413,4213)}t^{\des(\pi)}x^{\dd(\pi)}y^{\iar(\pi)}\\
&=yz+(y^2+txy)z^2+(y^3+2txy^2+2ty+t^2x^2y)z^3+\cdots.
\end{align*}
For any $e=(e_1,e_2,\ldots,e_n)\in\I_n(021)$, we always attach $e_{n+1}=n$ to the end of $e$. Let $\da(e):=|\{1< i\leq n: e_{i-1}<e_i<e_{i+1}\}|$ be the number of {\em double ascents} of $e$. Since the bijection $\Psi:\I_n(021)\rightarrow\S_n(2413,4213)$ transforms the set-valued statistics $\ASC$ to $\DES$, we have 
$$
G(t,x,y;z)=\sum_{n\geq1}z^n\sum_{e\in\I_n(021)}t^{\asc(e)}x^{\da(e)}y^{\izero(e)}. 
$$

\begin{lemma}\label{lem:eq021}
We have the algebraic functional equation for $G:=G(t,x,y;z)$:
\begin{equation}\label{funeq:021}
y^3z+(txy^2z+3y^3z-2y^2z-y^2)G+c_2G^2+c_3G^3=0,
\end{equation}
where 
\begin{align*}
c_2&:=2txy^2z-2txyz+3y^3z+tyz-4y^2z-2y^2+yz+2y\quad\text{and}\\
c_3&:=txy^2z-2txyz+y^3z+txz+tyz-2y^2z-tz-y^2+yz+t+2y-1.
\end{align*}
\end{lemma}
\begin{proof}
Let $\tilde \I_n(021)$ be the set of pairs $\e=(e,\phi)$, where $e\in\I_n(021)$ and  $\phi$ is an arbitrary function from $[r]$ to $\{0,1\}$ when $\izero(e)=r$. So $\tilde \I_n(021)$ can be viewed as $021$-avoiding inversion sequences of length $n$ whose initial zeros are $2$-colored. Let 
$$\tilde\I_n^{(0)}(021):=\{(e,\phi)\in\tilde\I_n(021): \phi(1)=0\}\quad \text{and} \quad \tilde\I_n^{(1)}(021):=\tilde\I_n(021)\setminus\tilde\I_n^{(0)}(021).$$ 
For each $\e=(e,\phi)\in\tilde \I_n(021)$ with $\izero(e)=r$, if $\e\in\tilde\I_n^{(0)}(021)$, then  define 
$$
\asc(\e):=\asc(e)+|\{i\in[r-1]: \phi(i)<\phi(i+1)\}|
$$
 and 
 $$\da(\e):=\da(e)+\chi(\phi(r-1)<\phi(r));
 $$
 otherwise, $\e\in\tilde\I_n^{(1)}(021)$ and we  define 
 $$
\asc(\e):=\asc(e)+|\{i\in[r-1]: \phi(i)<\phi(i+1)\}|+1
$$
 and 
 $$\da(\e):=\da(e)+\chi(\phi(r-1)<\phi(r))+\chi(e_2=1).
 $$
 The reason of defining these two statistics in this way will become transparent when we decompose $021$-avoiding inversion sequences. 
Let us introduce  two generating functions 
 $$
 \tilde G_0(t,x;z):=\sum_{n\geq1}z^n\sum_{\e\in\tilde\I_n^{(0)}(021)}t^{\asc(\e)}x^{\da(\e)}\quad \text{and}\quad \tilde G_1(t,x;z):=\sum_{n\geq1}z^n\sum_{\e\in\tilde\I_n^{(1)}(021)}t^{\asc(\e)}x^{\da(\e)}.
$$ 
 For convenience, we use the convention that $\tilde \I_0^{(0)}(021)$, $\tilde \I_0^{(1)}(021)$ and $\I_0(021)$ contain only the empty  inversion sequence. 
 
 Each $e=(e_1,\ldots,e_n)\in\I_n(021)$ with $k=\min\{i\in[n]:e_{i+1}=i\}$ can be decomposed into a pair $(\hat e,\e)$, where $\hat e:=(e_2,e_3,\ldots,e_k)\in\I_{k-1}(021)$ and $\e:=(\tilde e,\phi)\in\tilde\I_{n-k}^{(1)}(021)$ 
 such that 
 \begin{itemize}
 \item $\tilde e=(\tilde e_1,\tilde e_2,\ldots \tilde e_{n-k})$ with $\tilde e_{\ell}=e_{k+\ell}-k\cdot\chi(e_{k+\ell}>0)$ for $1\leq\ell\leq n-k$; 
 \item and $\phi(i)=\chi(e_{k+i}>0)$ for $1\leq i\leq \izero(\tilde e)$.
 \end{itemize}
 This decomposition is reversible and satisfies 
 \begin{align*}
 \asc(e)&=\asc(\hat e)+\asc(\e),\\
 \da(e)&=\da(\hat e)+\da(\e), \quad\text{and}\\
 \izero(e)&=\izero(\hat e)+1.
 \end{align*}
 Turning the above decomposition into generating function yields 
 \begin{equation}\label{eq:F021}
 G=yz(1+G)(1+\tilde G_1).
 \end{equation}

Similar decomposition as above for $2$-colored $021$-avoiding inversion sequences  gives the system of functional equations
 $$
\begin{cases}
\,\,\tilde  G_0=z(1+\tilde G_0+\tilde G_1)(1+\tilde G_1),
\\
\,\,\tilde G_1=z(tx+t\tilde G_0+tz(1-x)(1+\tilde G_1)+\tilde G_1)(1+\tilde G_1).
\end{cases}
$$
Eliminating $\tilde G_0$ gives the functional equation for $\bar G_1:=1+\tilde G_1$:
\begin{equation*}
 \bar G_1=1+(txz-2z)\bar G_1+(tz^2-2txz^2+z^2+2z)\bar G_1^2+(txz^3-tz^3+tz^2-z^2)\bar G_1^3.
\end{equation*}
On the other hand, solving~\eqref{eq:F021} gives 
$\bar G_1=\frac{G}{yz(1+G)}$. Substituting this expression into the above equation for $\bar G_1$ results in~\eqref{funeq:021}. 
\end{proof}

Next, we continue to compute the generating function for the left-hand side of~\eqref{eq:sep}:
\begin{align*}
S=S(t,x,y;z)&:=\sum_{n\geq1}z^n\sum_{\pi\in\S_n(2413,3142)}t^{\des(\pi)}x^{\dd(\pi)}y^{\iar(\pi)}\\
&=yz+(y^2+txy)z^2+(y^3+2txy^2+2ty+t^2x^2y)z^3+\cdots.
\end{align*}
This will be accomplished by applying Stankova's block decomposition~\cite{sk} (see also~\cite{lin}) of separable permutations that we now recall. 

\begin{lemma}[Stankova~\cite{sk}]\label{dec:stankova}
A permutation $\sigma\in\S_n$ is a separable permutation (i.e. avoids $2413$ and $3142$) if and only if:
\begin{itemize}
\item[(i)] $\sigma$ is of the form (positions of the blocks)
$$
A_1, A_2, \ldots, A_k,n, B_1, B_2,\ldots, B_l,\quad (|k-l|\leq1),
$$
where $A_1<A_2<\cdots<A_k$ and $B_1>B_2>\cdots >B_l$ are blocks with respect to $n$.
\item[(ii)] The elements in any block form a permutation that avoids both $2413$ and $3142$.
\end{itemize}
\end{lemma}

\begin{figure}
\begin{tikzpicture}[scale=0.8]
\draw (0,0.6) rectangle (1.5,1.2);
\draw (1.5,1.8) rectangle (3,2.4);
\draw (3,3) rectangle (4.5,3.6);
\draw (4.5,4.2) rectangle (6,4.8);

\draw (0.75,0.9) node{$A_1$};
\draw (2.25,2.1) node{$A_2$};
\draw (3.75,3.3) node{$\cdots$};
\draw (5.25,4.5) node{$A_k$};
\draw (6.7,5.2) node{$n$};

\draw (8.2,3.9) node{$B_1$};
\draw (9.7,2.7) node{$B_2$};
\draw (11.2,1.5) node{$\cdots$};
\draw (12.7,0.3) node{$B_l$};

\draw (7.4,4.2) rectangle (8.9,3.6);
\draw (8.9,3) rectangle (10.4,2.4);
\draw (10.4,1.8) rectangle (11.9,1.2);
\draw (11.9,0.6) rectangle (13.4,0);

\end{tikzpicture}
\caption{The block decomposition of separable permutations}
\label{block:sta}
\end{figure}
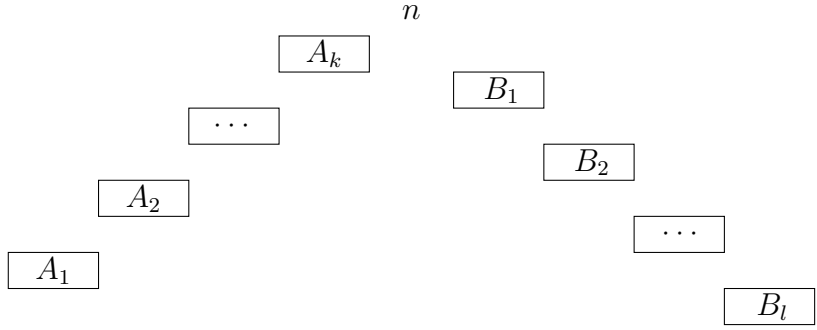

See Fig.~\ref{block:sta} for a transparent illustration of this lemma. Condition (ii) is clear, while condition (i) is equivalent to saying that $n$ is not an element of any subsequence of $\sigma$ that is order isomorphic to $2413$ or $3142$. Note that in the block decomposition, the minimal block can appear on either side of $n$. For example, compare the block decompositions of $259867431$ and $143867952$.

For convenience, we need to introduce  two variants of the double descents. Let 
$$
\dd_0(\pi):=|\{i\in[n]: \pi(i-1)>\pi(i)>\pi(i+1)\}|,
$$
where $\pi(0)=0$ and $\pi(n+1)=+\infty$, and 
$$
\dd_{\infty}(\pi):=|\{i\in[n]: \pi(i-1)>\pi(i)>\pi(i+1)\}|,
$$
where $\pi(0)=+\infty$ and $\pi(n+1)=0$. Let us introduce 
$$
L=L(t,x,y;z):=\sum_{n\geq1}z^n\sum_{\pi\in\S_n(2413,3142)}t^{\des(\pi)}x^{\dd_0(\pi)}y^{\iar(\pi)}=yz+(ty+y^2)z^2+\cdots
$$
and 
$$
R=R(t,x;z):=\sum_{n\geq1}z^n\sum_{\pi\in\S_n(2413,3142)}t^{\des(\pi)}x^{\dd_{\infty}(\pi)}=xz+(1+tx^2)z^2+\cdots.
$$
Set $B=B(y;z):=\frac{yz}{1-yz}$ and $\tilde L:=L-B$, where $B(y;z)$ enumerates identity permutations by length and $\iar$. 

\begin{lemma}
Let $S_1=S|_{y=1}$ and $L_1=L|_{y=1}$. We have the system of functional equations
\begin{equation}\label{sym:sepa}
\begin{cases}
\,\, L_1=\frac{S_1(1+tS_1)}{1+txS_1},\\
\,\,  R=\frac{S_1(S_1+x)}{1+S_1},\\
\,\, S_1=tS_1^3+tzS_1^2+(z+txz)S_1+z,\\
\,\, (1-z)\tilde L=\frac{tS_1B(1+B)}{1-tRB}+\frac{tzS_1\tilde L(2+L_1+B+tR-tRL_1B)}{(1-tRB)(1-tRL_1)},\\
\,\, S=\frac{B(1+tR)}{1-tRB}+\frac{z\tilde{L}(1+tR)^2}{(1-tRB)(1-tRL_1)}.
\end{cases}
\end{equation}
\end{lemma}
\begin{proof}
The first three equations of~\eqref{sym:sepa}  were proved by Wang~\cite{wan}. We begin with the proof of the fifth  equation in~\eqref{sym:sepa} by writing $S$ as an expression in $L$ and $R$. By Lemma~\ref{dec:stankova}, every  permutation $\pi\in\S_n(2413,3142)$  has block decomposition
$$
A_1, A_2, \ldots, A_k,n, B_1, B_2,\ldots, B_l,\quad (|k-l|\leq1),
$$
 where $A_1<A_2<\cdots<A_k$ and $B_1>B_2>\cdots >B_l$ are blocks with respect to $n$. 
 We distinguish three cases according to the pair $(k,l)$:
 \begin{itemize}
 \item[1)] $(k,l)=(j,j)$ ($j\geq1$). Permutations in this case contribute to $S$ the generating function 
 $$
 2yzB^j(tR)^j+2\sum_{i=1}^jzB^{i-1}\tilde L L_1^{j-i}(tR)^j.
 $$

 \item[2)] $(k,l)=(j+1,j)$ ($j\geq0$), and thus $1\in A_1$. Permutations in this case contribute to $S$ the generating function 
 $$
 yzB^{j+1}(tR)^j+\sum_{i=1}^{j+1}zB^{i-1}\tilde L L_1^{j+1-i}(tR)^j.
 $$

 \item[3)] $(k,l)=(j,j+1)$ ($j\geq0$), and thus $1\in B_l$.
 Permutations in this case contribute to $S$ the generating function 
 $$
 yzB^{j}(tR)^{j+1}+\sum_{i=1}^{j}zB^{i-1}\tilde L L_1^{j-i}(tR)^{j+1}.
 $$
 \end{itemize}


 Summing over all the above cases gives the fifth equation of~\eqref{sym:sepa}. The fourth equation of~\eqref{sym:sepa} is obtained by writing $L$ as an expression of $L$, $R$ and $S_1$ via the same block decomposition, the details of which are omitted due to the similarity. 
\end{proof}

We are ready to verify  Theorem~\ref{thm:sep}. 

\begin{proof}[{\bf Proof of  Theorem~\ref{thm:sep}}] We aim to verify that $S$ satisfies the same functional equation as $G$ in~\eqref{funeq:021}. From the first two equations of~\eqref{sym:sepa} we see that $L_1$ and $R$ are rational fractions in $S_1$. Thus, in view of the fourth equation of~\eqref{sym:sepa},  $\tilde L$ is also a rational fraction in $S_1$. Consequently, by the fifth equation of~\eqref{sym:sepa}, $S$ is a rational fraction in $S_1$ as well. Plugging the expressions for $L_1$, $R$ and $\tilde L$ into the fifth equation of~\eqref{sym:sepa} for $S$ and factoring out (using Maple) the rational fraction 
$$
y^3z+(txy^2z+3y^3z-2y^2z-y^2)S+c_2S^2+c_3S^3,
$$
where $c_2$ and $c_3$ are defined in Lemma~\ref{lem:eq021}, we see the factor $tS_1^3+tzS_1^2+(z+txz-1)S_1+z$ appears in the denominator (the resulting rational fraction is too long to be included here). This factor is zero due to the third equation of ~\eqref{sym:sepa}, which proves that $S$ satisfies the same functional equation as $G$ in~\eqref{funeq:021}. This completes the proof of Theorem~\ref{thm:sep}.
\end{proof}

\section{Conclusion}\label{sec:conclusion}

In this paper, we launch a systematic study of the Wilf-equivalence refined by two permutation statistics, namely $\comp$, the number of components, and $\iar$, the length of the initial ascending run, for all patterns (resp.~pairs of patterns) of length $3$. The results are summarized in Table~\ref{one 3-avoider} (resp.~Table~\ref{two 3-avoiders}), where the trivariate generating functions $\S(P)^{\des,\iar,\comp}(t,r,p;z)$ are supplied as well. In the cases where the pair $(\iar,\comp)$, together with other set-valued statistics, is symmetric over certain class of pattern-avoiding permutations, we construct various bijections to prove them (see e.g.~Theorems~\ref{thm:321=312}, \ref{thm:132-symm}, \ref{thm:213-231-conjugate}, and \ref{thm:132-312 and 132-321 symm}). On the other hand, our proof of the result concerning separable permutations (see Theorem~\ref{thm:sep}) is algebraic, and can hardly be called simple. Therefore, a direct bijection from $\S_n(2413,3142)$ to $\S_n(2413,4213)$ that preserves the statistics $\des$, $\dd$ and $\iar$ is much desired. 

In view of Lemmas~\ref{general:sym} and~\ref{gen:setv}, we pose the following open problem about a set-valued extension of Lemma~\ref{general:sym} for further investigation. 
\begin{?}\label{open:pro}
Let $\ST$ be a totally $\oplus$-compatible set-valued statistic. Let $P$ be a set of indecomposable patterns. Is it true that 
\begin{center} 
$|\S_n(P)^{\ST,\iar}|=|\S_n(P)^{\ST,\comp}|\Longleftrightarrow|\S_n(P)^{\ST,\iar,\comp}|=|\S_n(P)^{\ST,\comp,\iar}|$?
\end{center}
\end{?}
In particular, we suspect that Problem~\ref{open:pro} is true when $\ST$ is the statistic $\LMAX$. 
\begin{conj}
Let $P$ be a set of indecomposable patterns. Then
$$|\S_n(P)^{\LMAX,\iar}|=|\S_n(P)^{\LMAX,\comp}|\Longleftrightarrow|\S_n(P)^{\LMAX,\iar,\comp}|=|\S_n(P)^{\LMAX,\comp,\iar}|.
$$
\end{conj}

It is our hope, that the results presented and conceived (see also Conjecture~\ref{schroder:iar}) here, would attract more people to work on Wilf-equivalences refined by Comtet statistics, or to unearth and study new Comtet statistics in general.

\section*{Acknowledgement}

The second author was supported
by the National Science Foundation of China grants 11871247 and the project of Qilu Young Scholars of Shandong University.

\end{document}